\documentclass[a4paper,11pt]{amsart}

\usepackage{amsfonts, latexsym, amsmath, amssymb, amsthm, amscd}
\usepackage[all]{xy}
\usepackage[english]{babel}
\usepackage[utf8]{inputenc}
\usepackage{graphicx}
\usepackage{booktabs}
\usepackage{array, tabularx}
\usepackage{textcomp}
\usepackage{tikz}
\usepackage{multirow}
\usepackage{paralist}
\usepackage{comment}
\usepackage{url}
\usepackage{tensor}
\usepackage{MnSymbol}
\usepackage[hypertexnames=false,
backref=page,
    pdftex,
    pdfpagemode=UseNone,
    breaklinks=true,
    extension=pdf,
    colorlinks=true,
    linkcolor=blue,
    citecolor=blue,
    urlcolor=blue,
]{hyperref}

\usepackage{enumitem}
\usepackage[top=1.5in, bottom=.9in, left=0.9in, right=0.9in]{geometry}

\newcommand{\referenza}{}

\newtheorem{thm}{Theorem}[section]
\newtheorem*{thm*}{Theorem \referenza}
\newtheorem{cor}[thm]{Corollary}
\newtheorem*{cor*}{Corollary \referenza}
\newtheorem{lem}[thm]{Lemma}
\newtheorem*{lem*}{Lemma \referenza}

\newtheorem*{prop*}{Proposition \referenza}

\newtheorem*{conj*}{Conjecture \referenza}
\theoremstyle{definition}
\newtheorem{rmk}[thm]{Remark}
\newtheorem*{rmk*}{Remark}

\theoremstyle{definition}

\newtheorem{defi}[thm]{Definition}

\DeclareMathOperator{\End}{End}
\DeclareMathOperator{\Ker}{Ker}
\DeclareMathOperator{\tr}{tr}

\DeclareMathOperator{\im}{Im}
\DeclareMathSymbol{\Finv} {\mathord}{AMSb}{"60}
\DeclareMathOperator{\GL}{GL}

\newcommand\restrict[1]{\raisebox{-.5ex}{$|$}_{#1}}

\newcommand{\R}{\mathbb{R}}
\newcommand{\C}{\mathbb{C}}
\newcommand{\Z}{\mathbb{Z}}
\newcommand{\del}{\partial}
\newcommand{\delbar}{\overline{\partial}}

\numberwithin{equation}{section}

\newcommand{\A}{\mathcal{A}}
\let\c\overline
\let\phi\varphi
\newcommand{\de}[2]{\frac{\partial #1}{\partial #2}}
\DeclareFontFamily{U}{MnSymbolC}{}
\DeclareSymbolFont{MnSyC}{U}{MnSymbolC}{m}{n}
\DeclareFontShape{U}{MnSymbolC}{m}{n}{
    <-6>  MnSymbolC5
   <6-7>  MnSymbolC6
   <7-8>  MnSymbolC7
   <8-9>  MnSymbolC8
   <9-10> MnSymbolC9
  <10-12> MnSymbolC10
  <12->   MnSymbolC12}{}
\DeclareMathSymbol{\intprod}{\mathbin}{MnSyC}{'270}

\allowdisplaybreaks

\author{Tommaso Sferruzza}
\address[Tommaso Sferruzza]{
Dipartimento di Scienze Matematiche, Fisiche e Informatiche\\
Università di Parma}
\email{tommaso.sferruzza@unipr.it}

\title{Deformations of balanced metrics}

\thanks{}
%{53C55, 53C44, 32J15, 57T10} %16E45

\date{\today}

\begin{document}

\begin{abstract}
Small deformations of the complex structure do not always preserve special metric properties in the Hermitian non-K\"ahler setting. In this paper, we prove necessary conditions for the existence of smooth curves of balanced metrics $\{\omega_t\}_t$ which start with a fixed balanced metric $\omega$ for $t=0$, along a differentiable family of complex manifolds $\{M_t\}_t$.
\end{abstract}

\maketitle
\section{Introduction}
By a most celebrated theorem by Kodaira and Spencer in \cite{KS60}, small deformations of compact K\"ahler manifolds, i.e., compact complex manifolds $(M,J)$ admitting an Hermitian metric $g$ with associated fundamental form $\omega$ such that $d\omega=0$, are still K\"ahler. However, similar stability results do not hold in general for metric structures which naturally arise in the Hermitian setting and satisfy weaker conditions than the K\"ahler one.

Let $(M,J)$ be a $n$-dimensional compact complex manifold. A Hermitian metric $g$ on $(M,J)$ with fundamental associated form $\omega$ is said to be \emph{balanced}, or \emph{co-K\"ahler}, if $d\omega^{n-1}=0$. In particular, by the Leibniz rule, any K\"ahler metric is trivially balanced.
In \cite{AB90}, the existence of balanced metrics is proved to be non stable under small deformations. In that work, starting with the Iwasawa manifold $\mathbb{I}_3$ endowed with a natural balanced metric, the authors provide a complex curve of complex structures $J_t$ such that $(\mathbb{I}_3,J_t)$ does not admit any balanced metric, for small $|t|\neq0$. (For other results on balanced manifolds and deformations of balanced structures, see, for example, \cite{AU, FY, OUV, ST, RWZ2}).
Therefore, it is quite natural to study under which assumptions the existence of balanced metrics is preserved on the deformations of a compact complex non-K\"ahler balanced manifold.

Balanced metrics can be defined from the point of view of the torsion and, as for the K\"ahler case as in \cite{HL}, existence of balanced metrics on a compact complex manifold $(M,J)$ can be characterized via currents, see \cite{Mic}.
Let $(M,J,g,\omega)$ be a compact Hermitian manifold and let $\nabla^C$ and $T_g$ be respectively, the Chern connection associated to $g$ and its torsion tensor. The latter can be considered a $(2,0)$-form with values in the tangent bundle $TM$, or equivalently, as a $1$-form with values in $\End(TM)$. By classical results, the torsion tensor associated to a K\"ahler metric vanishes, therefore weaker conditions have been investigated. In fact, let us consider the real $1$-form $\tau_g:=\tr(T_g)$, called \emph{torsion} $1$-\emph{form} of $g$, obtained by tracing the torsion tensor $T_g$. If the condition $\tau_g=0$ holds, then the metric $g$ is called \emph{balanced}.
 
We note that balanced manifolds are a particular case of $p$-K\"ahler manifolds, i.e., manifolds admitting a real closed $(p,p)$-form $\Omega$ such that $\Omega$ defines a volume form once restricted to any $p$-dimensional submanifold, with $1\leq p\leq n-1$, where $n$ is the complex dimension of the manifold, see \cite{AB91}. Depending on the choice of $p$, two cases can be highlighted. If $p=1$, a 1-K\"ahler manifold is K\"ahler and also $p$-K\"ahler for every $1\leq p\leq n-1$. If instead $p=n-1$, a $(n-1)$-K\"ahler manifold is balanced. In \cite{AB90}, it is shown that the property of being $p$-K\"ahler, for any $1\leq p\leq n-1$, is not stable, in general, under small deformations of the complex structure. On the other hand, suitable conditions for the stability of $p$-K\"ahlerianity are given in \cite{RWZ}. 

%Under many aspects, the balanced condition is dual to the K\"ahler one. A simple example of this is the following. If $M$ is a K\"ahler manifold and if there exists a holomorphic immersion $f\colon N\rightarrow M$, then $N$ is K\"ahler, see \cite[Section 1]{Mic}. On the other hand, if $M$ is a balanced manifold and if there exists a holomorphic submersion $f\colon M \rightarrow N$, then $N$ is balanced. Therefore, the K\"ahler property is induced on submanifolds and the balanced property descends to quotients. 

In complex dimension two, the notions of being balanced and K\"ahler coincide, whereas there exist $n$-dimensional compact balanced manifolds which carry no K\"ahler metric, for $n\geq 3$. This is true, for example, of certain complex solvmanifolds.
In fact, the class of compact balanced manifolds contains the K\"ahler manifolds as well as many important categories of non-K\"ahler manifolds, including, for example, 1-dimensional families of K\"ahler varieties, the twistor spaces constructed from self-dual Riemannian 4-manifolds, and complex parallelisable manifolds.

In this paper, starting with compact complex balanced manifolds $(M,J,g,\omega)$, we address the problem of finding necessary cohomological conditions on $(M,J)$ in order that there exists a curve of balanced structures $\{J_t,g_t,\omega_t\}_{t\in(-\epsilon,\epsilon)}$  with $(J_0,g_0,\omega_0 )=(J,g,\omega)$.

More precisely, we obtain the following obstruction result, see Theorem \ref{thm:main}.
\begin{thm}\label{thm:main0}
Let $(M,J)$ be a $n$-dimensional compact complex manifold endowed with a balanced metric $g$ and associated fundamental form $\omega$. Let $\{M_t\}_{t\in I}$ be a differentiable family of compact complex manifolds with $M_0=M$ and parametrized  by $\phi(t)\in\A^{0,1}(T^{1,0}(M))$, for $t\in I:=(-\epsilon,\epsilon)$, $\epsilon>0$. Let $\{\omega_t\}_{t\in I}$ be a smooth family of Hermitian metrics along $\{M_t\}_{t\in I}$, written as
\begin{equation*}
\omega_t=e^{i_{\phi(t)}|i_{\overline{\phi(t)}}}\,\,(\omega(t)),
\end{equation*}
where, locally,  $\omega(t)=\omega_{ij}(t)\, dz^i\wedge d\overline{z}^j\in\mathcal{A}^{1,1}(M)$ and $\omega_0=\omega$.

If $\omega_t^{n-1}$ has local expression $e^{i_{\phi(t)}|i_{\overline{\phi(t)}}}(\omega_{i_1j_1}(t)\dots\omega_{i_{n-1}j_{n-1}}(t)\,dz^{i_1}\wedge d\overline{z}^{j_1}\wedge\dots\wedge dz^{i_{n-1}}\wedge d\overline{z}^{j_{n-1}})$, set
\[(\omega^{n-1}(t))':=\de{}{t}(\omega_{i_1j_1}(t)\dots\omega_{i_{n-1}j_{n-1}}(t))\, dz^{i_1}\wedge d\overline{z}^{j_1}\wedge\dots dz^{i_{n-1}}\wedge d\overline{z}^{j_{n-1}}\in\A^{n-1,n-1}(M).
\]
Then, if every metric $\omega_t$ is balanced, for $t\in I$, it must hold that
\begin{equation*}
\del\circ i_{\phi'(0)} (\omega^{n-1})=-\delbar (\omega^{n-1}(0))'.
\end{equation*}
\end{thm}
The map $e^{i_{\phi(t)}|i_{\overline{\phi(t)}}}\colon\A^{p,q}(M)\rightarrow\A^{p,q}(M_t)$, for any $p,q$ and $t\in(-\epsilon,\epsilon)$, is the real linear isomorphism between the space of $(p,q)$-forms on $M$ and the space of $(p,q)$-forms on $M_t$, referred to as \emph{extension map}, which was introduced in \cite{RZ} and we recall in section \ref{deformations}. Also, we denote by $i_{\psi}$ the interior product, or contraction, between $(p,q)$-forms and any $\psi\in \mathcal{A}^{0,1}(T^{1,0}(M))$, i.e.,  any $(0,1)$-differential form $\psi$ with values in $T^{1,0}(M)$; see section \ref{notations} for its definition.

As a direct consequence, Theorem \ref{thm:main0} yields the following obstruction regarding the Dolbeault cohomology group of $(M,J)$ of bi-degree $(n-1,n)$, see Corollary \ref{cor:main}.
\begin{cor}\label{cor:main0}
Let $(M,J)$ be a compact Hermitian manifold endowed with a balanced metric $g$ and associated fundamental form $\omega$. If there exists a smooth family of balanced metrics which coincides with $\omega$ in $t=0$, along the family of deformations $\{M_t\}_t$ with $M_0=M$ and parametrized by the $(0,1)$-vector form $\phi(t)$ on $M$, then the following equation must hold
\[
\left[\del\circ i_{\phi'(0)} (\omega^{n-1})\right]_{H_{\delbar}^{n-1,n}(M)}=0.
\]
\end{cor}
We remark that our results concern the existence of smooth families of balanced metrics $\{\omega_t\}_{t}$ along the differentiable family of complex manifold $\{M_t\}_{t}$, whereas classical stability concerns the general existence of balanced metrics on the family $\{M_t\}_{t}$.

In the proof of Theorem \ref{thm:main0}, we use formulas as in \cite{RZ}, for the action of the differentials $\del_t$ and $\delbar_t$ on $(p,q)$-forms on any element $M_t$ of a differentiable family of complex manifolds $\{M_t\}_t$, relying only on the complex differentials $\del_0=\del$ and $\delbar_0=\delbar$ of a fixed fiber $M_{0}$, and on the $(0,1)$-differential form with values in $T^{1,0}(M_0)$ which parametrizes the diffentiable family  $\{M_t\}_{t}$.
We point out that, in a similar manner, necessary conditions for the existence of smooth families of metrics along differentiable families of deformations can be proved for other special Hermitian metrics besides balanced ones, for example SKT metrics, as done in \cite{PS21}.

This note is organized as follows. In section \ref{notations}, we set the basic notions and definitions on complex manifolds which will be useful throughout the paper. In section \ref{deformations}, following the classical approach in \cite{KM}, we recall the results regarding existence and parametrization of deformations via a smooth $(0,1)$-form with values in the holomorphic tangent bundle, and the extension map mentioned above. In section \ref{main}, we recollect the formulas for the action of the complex differentials $\del_t$ and $\delbar_t$ on $(p,q)$-forms on a differentiable family of complex manifolds $\{M_t\}_t$, as introduced in \cite{RZ}, and we prove the main theorem. In section \ref{applications}, we apply Theorem \ref{thm:main0} and Corollary \ref{cor:main0}, yielding two examples, one for each family of $3$-dimensional complex parallelisable non-K\"ahler solvamanifolds (as classified in \cite{Nak}), namely, the complex parallelisable Nakamura manifold and the Iwasawa manifold. In particular, for each example we characterize, in terms of the parameters of the space of Kuranishi, smooth curves of deformations for which Corollary \ref{cor:main} provides obstructions.

\medskip\medskip
\noindent{\em Acknowledgments.} The author would like to sincerely thank Adriano Tomassini, both for his constant support and encouragement, and for many useful discussions and suggestions.

\section{Notations}\label{notations}
Let $(M,J,g,\omega)$ be a Hermitian manifold, i.e., a complex manifold $M$ endowed with an integrable almost-complex structure $J\in\End(TM)$ and a Hermitian metric $g$ on $M$ whose associated fundamental form $\omega$ is given by $\omega(X,Y)=g(JX,Y)$, for $X,Y\in TM$. Let $\dim_{\C}M=n.$

The complex structure $J$ induces the decomposition $\mathcal{A}_{\C}^k(M)=\oplus_{p+q=k}\mathcal{A}^{p,q}(M)$ on the spaces of complex $k$-differential forms into $(p,q)$-forms, for any $k\geqslant0$, so that the exterior differential $d$ acting on $(p,q)$-forms can be written as $d=\del+\delbar$, where $\del$ and $\delbar$ are the projections of $d(\mathcal{A}^{p,q}(M))$ onto, respectively, the spaces $\mathcal{A}^{p+1,q}(M)$ and $\mathcal{A}^{p,q+1}(M)$. 
With respect to this decomposition, the fundamental form $\omega$ is a $(1,1)$-form, i.e., $\omega\in\mathcal{A}^{1,1}(M).$

The Hermitian metric $g$ on $M$ is said to be \emph{balanced} if $d\omega^{n-1}=0.$
Since $\omega$ is a real, i.e., $\c{\omega}=\omega$, it can be easily seen that $g$ is balanced if, and only if, $\del\omega^{n-1}=0$, if, and only if, $\delbar\omega^{n-1}=0.$

%This condition can be written as $0=d\omega^{n-1}=(\del+\delbar)\omega^{n-1}=\del\omega^{n-1}+\delbar\omega^{n-1}$. Being $\del\omega^{n-1}\in\mathcal{A}^{n,n-1}(M)$ and $\delbar\omega^{n-1}\in\mathcal{A}^{n-1,n}(M)$, it follows that
%\[d\omega^{n-1}=0 \Leftrightarrow \begin{cases}\del\omega^{n-1}=0\\ \delbar\omega^{n-1}=0.\end{cases}
%\]
%Since $\omega$ is real, so are its powers. Therefore, we have that $\del\omega^{n-1}=\del\overline{\omega^{n-1}}=\overline{(\delbar\omega^{n-1})}$. Hence, we can conclude that
%\[g\quad \text{is balanced}\quad \Leftrightarrow \quad\del\omega^{n-1}=0\quad \Leftrightarrow\quad \delbar\omega^{n-1}=0.
%\]
 
If $\pi\colon E\rightarrow M$ is a complex vector bundle of rank $r$ over the $n$-dimensional compact Hermitian manifold $(M,J,g,\omega)$, we denote by $\bigwedge^{p,q}(M,E):=\bigwedge^{p,q}(M)\otimes E$ the bundle of the $(p,q)$-differential forms on $M$ with values in $E$ and by $\mathcal{A}^{p,q}(M,E):=\Gamma(M,\bigwedge^{p,q}(M,E))$ the space of its global smooth sections.

Let $\ast$ be the $\C$-antilinear Hodge operator on $(M,J,g,\omega)$ with respect to $g$. If $h$ is a Hermitian metric on the complex vector bundle $\pi\colon  E\rightarrow M$,  i.e., a smooth Hermitian scalar product on each fiber of $E$, we identify $h$ as a $\C$-antilinear isomorphism between $E$ and its dual $E^*$, see, for example, \cite{Huy04}. Then, we can consider the $\C$-antilinear Hodge $\ast_{E}$-operator
\begin{gather*}
{\ast}_E\colon\textstyle\mathcal{A}^{p,q}(M,E)\rightarrow\mathcal{A}^{n-p,n-q}(M,E^{\ast}),\\
{\ast}_E(\alpha\otimes s):=\ast({\alpha})\otimes h(s),\nonumber
\end{gather*}
for any $\alpha\otimes s\in\textstyle\mathcal{A}^{p,q}(M,E)$.
We recall the action of the Dolbeault operator $\delbar_E$ on $\mathcal{A}^{p,q}(M,E)$
\begin{equation}\label{eq:dbarE}
\delbar_E(\psi):=\sum\delbar(\alpha_i)\otimes s_i,
\end{equation}
for any $\psi\in\mathcal{A}^{p,q}(M,E)$, locally written as $\psi=\sum \alpha_i\otimes s_i$, with $\alpha_i\in\mathcal{A}^{p,q}(M)$ and $(s_1\dots,s_r)$ a local trivialization of $E$. The Dolbeault cohomology with values in the complex vector bundle $E$ is then defined as
\begin{equation*}
H^{p,q}_{\delbar_E}(M,E):=\displaystyle\frac{\Ker(\delbar_E\colon\mathcal{A}^{p,q}(M,E)\rightarrow\mathcal{A}^{p,q+1}(M,E))}{\im(\delbar_E\colon\mathcal{A}^{p,q-1}(M,E)\rightarrow\mathcal{A}^{p,q}(M,E))}.
\end{equation*}
The ${\ast}_E$-operator allows us to express $\delbar_E^*$ as
\begin{equation}\label{eq:dbar*E}
\delbar_E^{\ast}:=-{\ast}_{E^*}\circ\delbar_{E^*}\circ{\ast}_E
\end{equation}
and hence, the second order elliptic Laplacian operator
\begin{gather*}
\Delta_E:=\delbar_E^{\ast}\delbar_E+\delbar_E\delbar_E^{\ast}
\end{gather*}
and its harmonic forms
\begin{gather*}
\mathcal{H}^{p,q}(M,E):=\{\,\beta\in\textstyle\mathcal{A}^{p,q}(M,E):\Delta_E(\beta)=0\,\}.
\end{gather*}
The hermitian metric $h$  on $E$ induces the following Hermitian product on every $\mathcal{A}^{p,q}(M,E)$
\begin{equation*}
\llangle\alpha,\beta\rrangle:=\int_M h(\alpha,\beta)\ast 1.
\end{equation*}
With respect to $\llangle\cdot,\cdot\rrangle$, the operator $\delbar_E^*$ is the adjoint of $\delbar_E$ and the operator $\Delta_E$ is self-adjoint. We recall that Hodge theory yields the following decompositions for the spaces $\mathcal{A}^{p,q}(M,E)$
\begin{gather*}
\textstyle\mathcal{A}^{p,q}(M,E)=\mathcal{H}^{p,q}(M,E)\oplus \Delta_E(\mathcal{A}^{p,q}(M,E)).
\end{gather*}
Moreover, each space $\mathcal{H}^{p,q}(M,E)$ is finite-dimensional and projects bijectively onto the corresponding $H_{\delbar_E}^{p,q}(M,E)$, which also is finite-dimensional.

When the context is clear, we will omit the dependance on the vector bundle $E$ and on the manifold $M$ in the symbols for, respectively, operators and spaces of sections of differential forms with values in a vector bundle, i.e., for example, $\delbar$ and $\mathcal{A}^{p,q}(E)$.

We will refer to elements of $\A^{0,q}(T^{1,0}M)$ as \emph{$(0,q)$-vector forms} (on $M$). If $\phi=\beta\otimes V$ is a $(0,1)$-vector form, i.e, $\beta\in\A^{0,1}M$, $V\in T^{1,0}M$, we define the contraction by $\phi$ as the linear map
\begin{gather*}
i_{\phi}\colon\A^{p,q}(E)\rightarrow\A^{p-1,q+1}(E)\\
i_{\phi}(\alpha\otimes s):=\beta\wedge i_{V}(\alpha)\otimes s,
\end{gather*}
where $i_V(\alpha)$ is the usual interior product of a vector field and a $(p,q)$-differential form. 
In an analogous manner, we define $i_{\c\phi}(\alpha\otimes s):=\c\beta\wedge i_{\c{V}}(\alpha)\otimes s$ for the conjugate $\c\phi=\c\beta\otimes\c{V}$. For a $(0,1)$-vector form $\phi=\beta\otimes V$, we also define the contraction with $(0,1)$-vectof fields
\begin{gather*}
i_{\phi}\colon\Gamma(T^{0,1}M)\rightarrow\Gamma(T^{1,0}M)\\
i_{\phi}W:=\beta(W)V,
\end{gather*}
and we set $i_{\c\varphi}\c{W}:=\c\beta(\c{W})\c{V}$.
We will denote the map $i_{\phi}$ also by the symbol $\phi\intprod$.

\section{Preliminaries of deformation theory}\label{deformations}
For a more comprehensive view, we recollect the fundamental facts of deformation theory on compact complex manifolds.

Let us consider a domain $B$ of $\R^m$ (resp. of $\C^m$) and $\{M_t\}_{t\in B}$ a family of compact complex manifolds.

\begin{defi}\label{def:def}
We say that $\{M_t\}_{t\in B}$ is a \emph{differentiable} (resp. \emph{holomorphic}, or \emph{complex analytic}) \emph{family} 
%and that $M_t$ \emph{depends differentiably} (resp. \emph{holomorphically}) on $t\in B$
if there exist a differentiable (resp. complex) manifold $\mathcal{M}$ and a differentiable (resp. holomorphic) map $\pi$ from $\mathcal{M}$ onto $B$ which is proper and such that:
\begin{enumerate}
\item each fiber $\pi^{-1}(t)=M_t$ as complex manifolds, for every $t\in B$,
\item the rank of the Jacobian of $\pi$ coincides with the dimension (resp. complex dimension) of $B$, at each point of $\mathcal{M}$.\label{fam_def}
\end{enumerate}
\end{defi}
We note that from (\ref{fam_def}) of the definition, for  every $t\in B$, the fiber $\pi^{-1}(t)$ is a submanifold (resp. complex submanifold) of $\mathcal{M}$.
(We will adopt also the notation $(\mathcal{M},\pi,B)$ to denote the differentiable (resp. complex analytic) family $\{M_{t}\}_{t\in B}$.)
%\begin{defi}
%If $M$, $N$ are compact  complex manifolds, we say that $M$ \emph{is a differentiable} (resp. \emph{holomorphic}) \emph{deformation of $N$} if there exists a differentiable (resp. holomorphic) family $\{M_t\}_{t\in B}$ over a domain $B$ of $\R^m$ (resp. $\C^m$), with $M_{t_0}=M$, $M_{t_1}=N$ for some $t_0,t_1\in B$.
%\end{defi}

If $\{M_t\}_{t\in B}$ is a differentiable family of complex manifolds,
%the differentiable manifolds underlying any two elements of the family $M_{t_1}$ and $M_{t_2}$, $t_1,t_2\in B$, are known to be diffeomorphic
by a classical results of Ehresmann, see \cite{E47} or \cite[Proposition 6.2.2]{Huy04}, as a differentiable manifold, $\mathcal{M}$ can be regarded as the product
\begin{equation}\label{eq:ehr_thm}
\mathcal{M}\simeq M_{t_0}\times B,
\end{equation}
for a fixed $t_0\in B$; the manifold $M_{t_0}$ is referred to as the \emph{central fiber}. For the sake of simplicity, in what follows we will assume $t_0=0$ and $B=B(0,1)\subset\R^m$, i.e., $B=\{t\in\R^m : |t|<1 \}$.

Let $(\mathcal{M},\pi,B)$ be a differentiable family of compact complex manifolds. We can consider a system of local coordinates $\{\mathcal{U}_j,(\zeta_j,t)\}$ of $\mathcal{M}$, such that each $\mathcal{U}_j\simeq U_j\times B$, where
\begin{equation*}
U_j=\{(\zeta_j(p)): |\zeta_j(p)|<1\}.
\end{equation*}
The transition functions $f_{jk}$  identify points in $\mathcal{U}_j\cap\mathcal{U}_k\neq\emptyset$ by $\zeta_k=f_{jk}(\zeta_j,t)$.
By (\ref{eq:ehr_thm}), we can describe local coordinates of $\mathcal{U}_j$ as differentiable functions of coordinates of the central fiber $M_0=\pi^{-1}(0)$
\begin{equation}\label{eq:z_zeta}
\zeta_j=\zeta_j(\mathbf{z},t),
\end{equation}
where $\mathbf{z}=(z_j)_j$ are local coordinates on $M_0$. We note that each $\zeta_j(\mathbf{z},t)$ is a differentiable function of $(\mathbf{z},t)$, whereas, it depends holomorphically on $\mathbf{z}$ once $t$ is fixed.

Via coordinate expressions (\ref{eq:z_zeta}) and transition functions $f_{jk}$, in \cite[Chapter 4, Proposition 1.2]{KM}, Kodaira proves that the complex structure on each $M_t$, $t\in B$, can be parametrized by means of a smooth $(0,1)$-vector form $\phi(t)$ on $M_0$. More precisely, with respect to such $\phi(t)$,  the (local) holomorphic functions on each $M_t$ are defined as the differentiable complex valued functions $f$ on open sets of $M_0$ which satisfy
\begin{equation*}
\left(\delbar-i_{\phi(t)}\right)f(z,t)=0,
\end{equation*}
i.e., the holomorphic coordinates and, therefore, the complex structure on each $M_t$, for $t$ small enough, are determined by the $(0,1)$-vector form $\phi(t)$ on $M_0$.

%Let us set $\A_{q}:=\mathcal{A}^{0,q}(T^{1,0}M_0)$ and let $\Psi=\sum\psi^{i}\del_{i}$ and $\Xi=\sum\xi^{i}\del_{i}$ be respectively a $(0,p)$-vector form and a $(0,q)$-vector forms, where $\del_{i}:=\de{}{z^i}$. Then,
%\begin{equation}\label{eq:bracket}
%[\Psi,\Xi]:=\sum_{i,j=1}^n\big(\psi^{i}\wedge\del_{i}\xi^{j}-(-1)^{pq}\xi^{i}\wedge\del_{i}\psi^{j}\big)\del_{j}\,\,\in\A_{p+q}.
%\end{equation}
%In particular, $[\cdot,\cdot]$ is bilinear and satisfies the following properties, for any $\Psi\in\A_p$, $\Xi\in\A_q$ and $\Phi\in\A_r$:
%\begin{enumerate}
%\item $[\Psi,\Xi]=-(-1)^{pq}[\Xi,\Psi]$,
%\item $\delbar([\Psi,\Xi])=[\delbar\Psi,\Xi]+(-1)^p[\Psi,\delbar\Xi]$,
%\item $(-1)^{pr}[\Psi,[\Xi,\Phi]]+(-1)^{qp}[\Xi,[\Phi,\Psi]]+(-1)^{rq}[\Phi,[\Psi,\Xi]]$=0.
%\end{enumerate}
Furthermore, once a suitable bracket is defined on the spaces $\mathcal{A}^{0,q}(T^{1,0}M_0)$ (as in, for example, \cite[Chapter 6, Section 1]{Huy04}), the deformations of the complex structure on a compact complex manifold can be represented  according to the following theorem, see \cite[Chapter 4, Theorem 1.1]{KM}.
\begin{thm}\label{thm:Kod-psi}
If $(\mathcal{M},\pi,B)$ is a differentiable family of compact complex manifolds, then the complex structure on each $M_t=\pi^{-1}(t)$ is represented by a $(0,1)$-vector form $\phi(t)$ on $M_0$, such that $\phi(0)=0$ and
\begin{equation}\label{eq:MC-eq}
\delbar\phi(t)-\frac{1}{2}[\phi(t),\phi(t)]=0\qquad\textit{(Maurer-Cartan equation).}
\end{equation}
%\item $\frac{\del M_t}{\del t}|_{t=0}$ corresponds to $\eta=-\frac{\del\Psi(t)}{\del t}|_{t=0}$.
\end{thm}

A more general theory, known as \emph{Kuranishi} theory, assures that deformations exist on any compact complex manifold. We recall the fundamental result.

\begin{comment}
First, let us a define the notion of complete family of deformations.
\begin{defi}\label{def:complete_fam}
If $M$ is a complex manifold, $\pi\colon\mathcal{M}\rightarrow B$ a complex analytic family over $B$, we say that $(\mathcal{M},B,\pi)$ is \emph{complete} at $b\in B$ if for any family $(\mathcal{N},A,\rho)$ such that $\rho^{-1}(a)=\pi^{-1}(b)=M_b$ there is a nieghborhood $U\ni a$ and holomorphic maps $\Phi\colon\rho^{-1}(U)\rightarrow\mathcal{M}$, $h\colon U\rightarrow B$ such that
\begin{enumerate}
\item the following diagram
\begin{equation*}
\xymatrix{ 
\rho^{-1}(U) \ar[r] \ar[d]_{\rho} &   \mathcal{M} \ar[d]^{\pi}\\
 U \ar[r]^{h} & B}
\end{equation*}
commutes;
\item $\Phi$ maps $\rho^{-1}(s)$ biholomorphically onto $\pi^{-1}(h(s))$ for each $s\in U$;
\item $\Phi\colon\pi^{-1}(b)=M_b\rightarrow M_b$ is the identity map.
\end{enumerate}
\end{defi}
\end{comment}

Let $(M,J,g)$ be a compact Hermitian manifold. For the spaces $\mathcal{A}^{0,q}(T^{1,0}(M))$, by Hodge theory there exist the direct sum decompositions
\begin{equation}
\mathcal{A}^{0,q}(T^{1,0}(M))=\mathcal{H}^{0,q}(M,T^{1,0}(M))\oplus\Delta_{T^{1,0}}\left(\mathcal{A}^{0,q}(T^{1,0}(M))\right),
\end{equation}
where $\Delta_{T^{1,0}(M)}$ is the Laplacian operator acting on $\A^{0,q}(T^{1,0}(M))$ and $\mathcal{H}^{0,q}(M,T^{1,0}(M))$
is the space of $\Delta_{T^{1,0}(M)}$-harmonic $(0,q)$-vector forms on $M$.
%and denote the Hermitian extension of $h$ to $\mathcal{A}_q$ by the  symbol $h$. For every $\Psi,\Xi\in\mathcal{A}_q$, the bilinear form
%\[
%\llangle\Psi,\Xi\rrangle:=%\int_M h(\Psi,\Xi)*1,
%\]
%defines an inner product on $\mathcal{A}_q$, with $\ast$ the $\C$-antilinear Hodge operator. With respect to $\llangle \cdot,\cdot\rrangle$, the Laplacian operator $\square$ on $\mathcal{A}_q$, defined by
%\[
%\square:=\delbar^*\delbar+\delbar\delbar^*,
%\]
%is self-adjoint and the operator $\delbar^*$ is the adjoint operator of $\delbar$. If we denote the space of harmonic forms with respect to the Laplacian operator $\square$ by
%\[
%\mathcal{H}^q:=\{\Psi\in\mathcal{A}_q:\square\Psi=0\},
%\]
%then, Hodge theory induces a decomposition on the each space $\mathcal{A}_q$ as a direct sum of orthogonal subspaces
%\[
%\mathcal{A}_q=\mathcal{H}^q\oplus\square\mathcal{A}_q.
%\]

With respect to this decomposition, both the projection
\[
G\colon\mathcal{A}^{0,q}(T^{1,0}(M))\rightarrow
\Delta_{T^{1,0}(M)}\left(\mathcal{A}^{0,q}(T^{1,0}(M))\right),
\]
known as the \emph{Green operator}, and the projection map
\[
H\colon\mathcal{A}^{0,q}(T^{1,0}(M))\rightarrow \mathcal{H}^{0,q}(M,T^{1,0}(M)),
\]
are well defined, see \cite[Chapter 4, Lemma 2.1]{KM}.

\begin{thm}[Kuranishi]\label{thm-Kur}
Let $M$ be a compact complex manifold and let $\{\eta_{i}\}$ be a base for the space $\mathcal{H}^{0,1}(M,T^{1,0}(M))$. Let $\phi(t)\in\mathcal{A}^{0,1}(T^{1,0}(M))$ be a power series solution of the equation
\begin{equation}\label{eq:thm-Kur}
\phi(t)=\eta(t)+\frac{1}{2}\delbar^* G[\phi(t),\phi(t)],
\end{equation}
where $\eta(t)=\sum_{i=1}^m t_{i}\eta_{i}$, $t=(t_1,\dots,t_m)\in B_r(0)\subset\C^m$, $r>0$, and let $S=\{t\in B_r(0): H[\phi(t),\phi(t)]=0\}$. 
Then for each $t\in S$, $\phi(t)$ determines a complex structure $M_t$ on the differentiable manifold underlying $M$.
\end{thm}
The space of parameters $S$, referred to as the \emph{space of Kuranishi}, %The proof of Theorem (\ref{thm-Kur}) shows that a $(0,1)$-vector form $\phi(t)$ on $M$ which satisfies equation (\ref{eq:thm-Kur}) can be constructed as a converging power series
%\[
%\phi(t)=\sum_{i=1}^{\infty}\phi_{i}(t)
%\]
%in which the forms
%\[
%\phi_{i}(t)=\sum_{j_{1}+\dots+j_{m}=i}\phi_{j_1\dots j_m}t_1^{j_1}\cdots\, t_m^{j_m},\quad \phi_{j_1\dots j_m}\in\mathcal{A}_1,
%\]
%are determined via a recursive formula. In fact, if $\{\eta_{j}\}_{j=1}^n$ is a basis for $\mathcal{H}^1$ and we set $\psi_1(t)=\sum_{j=1}^i t_{j}\eta_{j}$, 
%equation (\ref{eq:thm-Kur}) assures that each term $\phi_{i}$ can be computed as
%\begin{equation}\label{eq:rec-Kur}
%\phi_{i}(t)=\frac{1}{2}\delbar^* G\,\Big(\,\sum_{\kappa=1}^{i-1}\,\,[\phi_{\kappa}(t),\phi_{i-\kappa}(t)]\,\Big).
%\end{equation}
in general can have singularities and, hence, it may not have a structure of smooth manifold. However, the family $\{M_t\}_{t\in S}$ can still be regarded as a complex analytic family, as proved in \cite{Kur65}.

For a differentiable family of compact complex manifolds, it is useful to understand the decompositions of the complexified cotangent bundle and its powers on each fiber of the family (see also \cite{Tu} for a generalization of such decompositions).

Let us suppose that a differentiable family $(\mathcal{M},\pi,B)$ is parametrized by the $(0,1)$-vector form $\phi(t)$ on the central fiber $M_0=:M$. For the sake of semplicity, we suppose that $B\subset\R$, i.e., $B=(-\epsilon,\epsilon)$, for $\epsilon>0$.

Let $i_{\phi(t)}^k$ be the interation of $\phi(t)$ for $k$ times
%i.e., $i_{\phi(t)}^k:=\underbrace{i_{\phi(t)}\circ\dots\circ i_{\phi(t)}}_{k\,\,\text{times}}$,
and $\overline{\phi(t)}\in\A^{1,0}(T^{0,1}M)$ be the conjugate of $\phi(t)$. Then, the following operators
\begin{equation*}
e^{i_{\phi(t)}}:=\sum_{k=0}^{\infty}\frac{1}{k!}i_{\phi(t)}^k, \qquad\qquad\qquad e^{i_{\c{\phi(t)}}}:=\sum_{k=0}^{\infty}\frac{1}{k!}i_{\overline{\phi(t)}}^k
\end{equation*}
are well-defined since each summation is finite, being $\dim_{\C}M$ finite.

As in \cite[Definition 2.8]{RZ}, for any $\alpha\in\mathcal{A}^{p,q}(M)$ with local expression $\alpha=\alpha_{i_1\dots i_p j_1\dots j_q}dz^{i_1}\wedge\dots\wedge dz^{i_p}\wedge d\overline{z}^{j_1}\wedge\dots\wedge d\overline{z}^{j_q}$, the \emph{extension map} is defined as
\begin{equation}\label{def-exp}
e^{i_{\phi(t)}|i_{\overline{\phi(t)}}}(\alpha):=\alpha_{i_1\dots i_p j_1\dots j_q}e^{i_{\phi(t)}}(dz^{i_1}\wedge\dots\wedge dz^{i_p})\wedge e^{i_{\overline{\phi(t)}}}(d\overline{z}^{j_1}\wedge\dots\wedge d\overline{z}^{j_q}).
\end{equation}
Such map allows to estabilish a correspondence between the $(p,q)$-forms on the central fiber $M$ and on any $M_t$, as proved in the following lemma, see \cite[Lemma 2.9, 2.10]{RZ}.
\begin{lem}\label{thm-exp}
For any $p,q$, and for $t$ small, the extension map $e^{i_{\phi(t)}|i_{\overline{\phi(t)}}}\colon\A^{p,q}(M)\rightarrow\A^{p,q}(M_t)$ is a real linear isomorphism.
\end{lem}
Moreover, each space of complex $k$-differential forms can be decomposed as
\begin{equation}\label{eq:decomp_t}
\A_{\C}^k(M)=\oplus_{p+q=k}\A^{p,q}(M_t),\qquad \text{for}\,\,\,k\in\{1,\dots,n\}.
\end{equation}

\begin{comment}
\begin{rmk}\label{rmk-int}
For a $(0,1)$-vector form $\phi(t)$ on $M$ such that $\phi(0)=0$, satisfying the Maurer-Cartan equation (\ref{eq:MC-eq}) is equivalent to the condition of integrability for the complex structure $J_t$ on $M_t$, i.e.,
\begin{equation}\label{eq:integr}
(d\alpha)^{0,2}=0 \qquad \forall\alpha\in\A^{1,0}(M_t),
\end{equation}
where $(d\alpha)^{0,2}$ is the projection onto $\mathcal{A}^{0,2}(M_t)$ of the form $d\alpha\in\mathcal{A}_{\C}^2(M)$, with regard to to decomposition (\ref{eq:decomp_t}). In fact, Lemma \ref{thm-exp} implies that the map $(I-\phi)\intprod:\Gamma(T^{0,1}M)\to \Gamma(T^{0,1}M_t)$ is an isomorphism for $t$ small, and,  for $X,Y\in \Gamma(T^{1,0}M)$,
\begin{equation*}
-d(e^{i_{\phi(t)}|i_{\overline{\phi(t)}}}(\alpha))(X-\phi(t)(X),Y-\phi(t)(Y))=e^{i_{\phi(t)}|i_{\overline{\phi(t)}}}(\alpha)\left((\delbar\phi(t)-\frac{1}{2}[\phi(t),\phi(t)])(X,Y)\right).
\end{equation*}
See also \cite[Proposition 6.1.2]{Huy04}.
Furthermore, for a $(0,1)$-vector form satisfying (\ref{eq:thm-Kur}), the defining property of $S$, i.e., $H[\phi(t),\phi(t)]=0$, is equivalent to the integrability condition given by the Maurer-Cartan equation (\ref{eq:MC-eq}) (see \cite[Chapter 4, Proposition 2.5]{KM}).
\end{rmk}
\end{comment}

\section{Main Theorem}\label{main}
Let $\{M_t\}_{t\in I}$ be a differentiable family of compact complex manifolds parametrized by a $(0,1)$-vector form $\phi(t)$ on $M:=M_{0}$,  where we assume $t\in I$, $I:=(-\epsilon,\epsilon)$, $\epsilon>0$.
On each fiber $M_t$ of $\{M_t\}_{t\in I}$, accordingly to decompositions (\ref{eq:decomp_t}), the differential operators $\del_t$ and $\delbar_t$ are, by definition,
\begin{gather*}
\del_t:=\pi_t^{p+1,q}\circ d\colon \A^{p,q}(M_t)\rightarrow\A^{p+1,q}(M_t),\\
\delbar_t:=\pi_t^{p,q+1}\circ d\colon \A^{p,q}(M_t)\rightarrow\A^{p,q+1}(M_t),
\end{gather*}
where $\pi_t^{p+1,q}$ and $\pi_t^{p,q+1}$ are the projections of $d(\A^{p,q}(M_t))$ onto, respectively, $\A^{p+1,q}(M_t)$ and $\A^{p,q+1}(M_t)$, for any $p$, $q$.

Following \cite[Proposition 2.7]{RZ}, we have the formulas for $\del_t$ and $\delbar_t$ acting on any differentiable complex function $f$ on $M_t$
\begin{gather*}
\del_t f=e^{i_{\phi}}\Big((I-\phi\overline{\phi})^{-1}\intprod(\del-\overline{\phi}\intprod\delbar)f\Big),\\
\delbar_t f=e^{i_{\overline{\phi}}}\Big((I-\overline{\phi}\phi)^{-1}\intprod(\delbar-\phi\intprod\del)f \Big),
\end{gather*}
where we use the abbreviations $\phi\overline{\phi}:=\overline{\phi}\intprod\phi$,  $\overline{\phi}\phi:=\phi\intprod\overline{\phi}$ and $\phi:=\phi(t)$, whereas the action of $\del_t$ and $\delbar_t$ on $(p,q)$-differential forms $e^{i_{\phi}|i_{\overline{\phi}}}\alpha\in\A^{p,q}(M_t)$, with $\alpha\in\A^{p,q}(M)$, is proved in \cite[Proposition 2.13]{RZ} to be
\begin{align}
\del_t(e^{i_{\phi}|i_{\overline{\phi}}}\alpha)&=e^{i_{\phi}|i_{\overline{\phi}}}\Big((I-\phi\overline{\phi})^{-1}\Finv([\delbar,i_{\overline{\phi}}]+\del)(I-\phi\overline{\phi})\Finv \alpha\Big)%\nonumber\\
%&=(I+\phi+\overline{\phi})\Finv\Big((I-\phi\overline{\phi})^{-1}\Finv([\delbar,i_{\overline{\phi}}]+\del)(I-\phi\overline{\phi})\Finv \alpha\Big)
,\label{delt}\\
\delbar_t(e^{i_{\phi}|i_{\overline{\phi}}}\alpha)&=e^{i_{\phi}|i_{\overline{\phi}}}\Big((I-\overline{\phi}\phi)^{-1}\Finv([\del,i_{\phi}]+\delbar)(I-\overline{\phi}\phi)\Finv\alpha\Big)%\nonumber\\
%&=(I+\phi+\overline{\phi})\Finv\Big((I-\overline{\phi}\phi)^{-1}\Finv([\del,i_{\phi}]+\delbar)(I-\overline{\phi}\phi)\Finv\alpha\Big)
.\label{delbart}
\end{align}
Let now $(M,J,g,\omega)$ be a compact Hermitian manifold of complex dimension $n$ endowed with a balanced metric $g$, i.e., $\delbar\omega^{n-1}=0$ and let $\{M_t\}_{t\in I}$ be a differentiable family  of deformations such that $M_0=M$, with $\{M_t\}_{t\in I}$ parametrized by a $(0,1)-$vector form $\phi(t)$ on $M$, for $t\in I=(-\epsilon,\epsilon)$, $\epsilon>0$. Let also $\{\omega_t\}_{t\in I}$ be a family of Hermitian metrics on $\{M_t\}_{t\in I}$, such that $\omega_0=\omega$ and we suppose the metrics $g_t$ to be balanced, i.e.,
\begin{equation}\label{eq-bal_t}
\delbar_t\omega_t^{n-1}=0, \quad\forall t\in I.
\end{equation}
We remark that, by Lemma \ref{thm-exp}, we can write each $\omega_t$ as $e^{i_{\phi}|i_{\overline{\phi}}}(\omega(t))$, where locally $\omega(t)=\omega_{ij}(t)dz^i\wedge d\overline{z}^j$. In particular, by definition of $e^{i_{\phi}|i_{\overline{\phi}}}$, it is easy to check that
\begin{align*}
\omega_t^{n-1}&=(e^{i_{\phi}|i_{\overline{\phi}}}(\omega(t)))^{n-1}=e^{i_{\phi}|i_{\overline{\phi}}}(\omega^{n-1}(t))\\
&=e^{i_{\phi}|i_{\overline{\phi}}}(f_v(t)\,dz^{i_1}\wedge d\overline{z}^{j_1}\wedge\dots\wedge dz^{i_{n-1}}\wedge d\overline{z}^{j_{n-1}}),
\end{align*}
where we denote $f_v(t):=\omega_{i_1j_1}(t)\dots \omega_{i_{n-1}j_{n-1}}(t)$, with $v=(i_1,j_1,\dots,i_{n-1},j_{n-1})$ and $i_k,j_k\in\{1,\dots,n\}$, $k=\{1,\dots,n-1\}$. 

We can then apply formula (\ref{delbart}) to (\ref{eq-bal_t}) and, by making use of Taylor series expansion and differentiating with respect to $t$ in $t=0$,  we are able to prove the main theorem.

As a final remark before the theorem, we note that, for any $(p,q)$-differential form $\alpha$ on $M$, locally written as $\alpha=\alpha_{i_1\dots i_p j_1\dots j_q}dz^{i_1}\wedge\dots\wedge dz^{i_p}\wedge d\overline{z}^{j_1}\wedge\dots\wedge d\overline{z}^{j_q}$, the simultaneous contraction, which we will denote by $\Finv$, on each form component of $\alpha$, i.e.,
\begin{align*}
\phi\Finv\alpha:=
\alpha_{i_1\dots i_p j_1\dots j_q}\phi\intprod dz^{i_1}\wedge \dots\wedge \phi\intprod dz^{i_p}\wedge \c\phi\intprod\overline{z}^{j_1}\wedge \dots \wedge \c\phi\intprod d\overline{z}^{j_q}.
\end{align*}
is a well-defined operator and the extension map $e^{i_\phi|i_{\overline{\phi}}}$ associated to a $(0,1)$-vector form $\phi(t)$, can be written, in terms of $\Finv$, as
\begin{equation}\label{eq:Finv_exp}
e^{i_{\phi}|i_{\overline{\phi}}}=(I+\phi+\overline{\phi})\Finv.
\end{equation}

\begin{thm}\label{thm:main}
Let $(M,J)$ be a $n$-dimensional compact complex manifold endowed with a balanced metric $g$ and associated fundamental form $\omega$. Let $\{M_t\}_{t\in I}$ be a differentiable family of compact complex manifolds with $M_0=M$ and parametrized  by $\phi(t)\in\A^{0,1}(T^{1,0}(M))$, for $t\in I:=(-\epsilon,\epsilon)$, $\epsilon>0$. Let $\{\omega_t\}_{t\in I}$ be a smooth family of Hermitian metrics along $\{M_t\}_{t\in I}$, written as
\begin{equation*}
\omega_t=e^{i_{\phi}|i_{\overline{\phi}}}\,\,(\omega(t)),
\end{equation*}
where, locally,  $\omega(t)=\omega_{ij}(t)\, dz^i\wedge d\overline{z}^j\in\mathcal{A}^{1,1}(M)$ and $\omega_0=\omega$.

If $\omega_t^{n-1}$ has local expression $e^{i_{\phi}|i_{\overline{\phi}}}(\omega_{i_1j_1}(t)\dots\omega_{i_{n-1}j_{n-1}}(t)\,dz^{i_1}\wedge d\overline{z}^{j_1}\wedge\dots\wedge dz^{i_{n-1}}\wedge d\overline{z}^{j_{n-1}})$, set
\[(\omega^{n-1}(t))':=\de{}{t}(\omega_{i_1j_1}(t)\dots\omega_{i_{n-1}j_{n-1}}(t))\, dz^{i_1}\wedge d\overline{z}^{j_1}\wedge\dots dz^{i_{n-1}}\wedge d\overline{z}^{j_{n-1}}\in\A^{n-1,n-1}(M).
\]
Then, if every metric $\omega_t$ is balanced, for $t\in I$, it must hold that
\begin{equation*}
\del\circ i_{\phi'(0)} (\omega^{n-1})=-\delbar (\omega^{n-1}(0))'.
\end{equation*}
\end{thm}

Given Theorem \ref{thm:main}, it is straightforward to see that the following cohomological obstruction holds.
\begin{cor}\label{cor:main}
Let $(M,J)$ be a compact Hermitian manifold endowed with a balanced metric $g$ and associated fundamental form $\omega$. If there exists a smooth family of balanced metrics which coincides with $\omega$ in $t=0$, along the family of deformations $\{M_t\}_t$ with $M_0=M$ and parametrized by the $(0,1)$-vector form $\phi(t)$ on $M$, then the following equation must hold
\[
\left[\del\circ i_{\phi'(0)} (\omega^{n-1})\right]_{H_{\delbar}^{n-1,n}(M)}=0.
\]
\end{cor}
\begin{proof}[Proof of Theorem \ref{thm:main}]
The metrics $\omega_t$ are balanced for every $t\in I$, i.e., $\delbar_t\omega_t^{n-1}=0$. By means of the extension map, this equation can be written as
\begin{equation*}
\delbar_t\left(e^{i_{\phi}|i_{\overline{\phi}}}(\omega^{n-1}(t))\right)=0,\quad\forall t\in I.
\end{equation*}
Also, formula (\ref{delbart}) implies that $\delbar_t\omega_t^{n-1}=0$ for every $t\in I$ if and only if
\begin{equation*}
e^{i_{\phi}|i_{\overline{\phi}}}\Big((I-\overline{\phi}\phi)^{-1}\Finv([\del,i_{\phi}]+\delbar)(I-\overline{\phi}\phi)\Finv(\omega^{n-1}(t))\Big)=0,\quad\forall t\in I.
\end{equation*}
We now use equation (\ref{eq:Finv_exp}) and we develop in Taylor series expansion centered in $t=0$ the term $\delbar_t\omega_t^{n-1}$, noting that
\begin{equation*}
\phi=\phi(t)=t\phi'(0)+o(t)
\end{equation*}
and, therefore,
\begin{equation*}
(I-\phi\overline{\phi})=(I-\overline{\phi}\phi)=(I-\phi\overline{\phi})^{-1}=(I-\overline{\phi}\phi)^{-1}=I+o(t),
\end{equation*}
to obtain
\begin{align*}
\delbar_t\omega_t^{n-1}&=\left(I+t\phi'(0)\intprod+t\overline{\phi'(0)}\intprod\right)\Finv\left(([\del,t\phi'(0)\intprod]+\delbar)(\omega^{n-1}(0)+t(\omega^{n-1}(0))')\right)+o(t)\\
&=\left(I+t\phi'(0)\intprod+t\overline{\phi'(0)}\intprod\right)\Finv\left(t\del(\phi'(0)\intprod \omega^{n-1}(0))+t\delbar(\omega^{n-1}(0))') \right)+o(t)\\
&=t\del(\phi'(0)\intprod \omega^{n-1}(0))+t\delbar(\omega^{n-1}(0))') +o(t)
\end{align*}
Since $\delbar_t\omega_t^{n-1}=0$ holds true for every $t\in I$, if we differentiate it with respect to $t$ in $t=0$, we obtain
\begin{equation*}\label{eq:thm_1}
\frac{\del}{\del t}\restrict{t=0}\left(\delbar_t\omega_t^{n-1}\right)=\frac{\del}{\del t}\restrict{t=0}\left[t\del(\phi'(0)\intprod \omega^{n-1}(0))+t\delbar(\omega^{n-1}(0))') +o(t)\right]=0.
\end{equation*}
Hence,
\begin{equation*}
\del(\phi'(0)\intprod \omega^{n-1})+\delbar(\omega^{n-1}(0))'=0,
\end{equation*}
therefore concluding the proof.
\end{proof}
\section{Applications}\label{applications}
In this section, we apply Theorem \ref{thm:main} and Corollary \ref{cor:main} to find obstructions on each family of non-K\"ahler complex parallelisable solvmanifolds as characterized in \cite{Nak}. In particular, we will focus on the complex parallelisable Nakamura manifold and the Iwasawa manifold. In the examples, we will refer to differentiable families over a real interval $I$ as \emph{curves of deformations}.

\subsection{Example 1}{\em (Complex parallelisable Nakamura manifold)}.
Let $G:=\C\ltimes_{\gamma}\C^2$ be the complex Lie group given by the action of $\C$ on $\C^2$, via
\begin{equation}
\gamma(z)=\begin{pmatrix}
e^z & 0\\
0 & e^{-z}
\end{pmatrix}.
\end{equation}
Let us consider the discrete subgroup $\Gamma$ of $G$ of the form $\Gamma:=\left(\Z(a+ib)+\Z(c+id)\right)\ltimes_{\gamma}\Gamma''$, where
\begin{itemize}
\item  the set $\Gamma''$ is a lattice of $\C^2$;
\item the complex numbers $a+ib$ and $c+id$  are such that $\Z(a+ib)+\Z(c+id)$ is a lattice in $\C$;
\item the matrices $\gamma(a+ib)$ and $\gamma(c+id)$ are conjugates in $SL(4;\Z)$, where we regard $SL(2;\C)\subset SL(4;\R)$.
\end{itemize}
Then $\Gamma$ is a lattice of $G$ and the compact quotient $M:=\Gamma/G$ is called the \emph{complex parallelisable Nakamura Manifold}, see \cite[Section 2]{Nak} for details on its construction.

%As embedded in $SL(4;\C)$, the group $G$ can be seen as
%\begin{equation*}
%G\simeq \left\{
%\begin{pmatrix}
%e^{z^1} & 0 & 0 & z^3\\
%0 & e^{-z^1} & 0 & z^2\\
%0 & 0 & 1 & z^1\\
%0 & 0 & 0 & 1
%\end{pmatrix}: z^1,z^2,z^3\in\C \right\}.
%\end{equation*}
It is well known that $G$ is a solvable non nilpotent Lie group, therefore the quotient $M$ is a $3$-dimensional solvamanifold, which is biholomorphic to $\C^3$.

If $\{z^1\}$ and $\{z^2,z^3\}$ are the standard coordinates on respectively $\C$ and $\C^2$, a left-invariant frame of $(1,0)$-vector fields on $G$ is given by $\{Z_1,Z_2,Z_3\}$, where
\begin{align*}
\begin{cases}
Z_1=\frac{\del}{\del z^1}\\
Z_2=e^{z^1}\frac{\del}{\del z^1}\\
Z_3=e^{-z^1}\frac{\del}{\del z^3}
\end{cases}
\end{align*}
and the dual coframe of $(1,0)$-differential forms in $\mathcal{A}^{1,0}(M)$ is given by $\{\eta^1,\eta^2,\eta^3\}$, where
\begin{align*}
\begin{cases}
\eta^1=dz^1\\
\eta^2=e^{-z^1}dz^2\\
\eta^3=e^{z^1}dz^3.
\end{cases}
\end{align*}
Note that structure equations
\begin{align}
\begin{cases}\label{Nak:struct_eq}
d\eta^1=0\\
d\eta^2=-\eta^1\wedge\eta^2\\
d\eta^3=\eta^1\wedge\eta^3.
\end{cases}
\end{align}
imply that the coframe of left-invariant $(1,0)$-forms $\{\eta^1,\eta^2,\eta^3\}$ induce an almost-complex left-invariant structure $J$ on $M$, which is integrable.

From now on, we adopt the abbreviation for the wedge product of differential forms, i.e., for example, $\eta^{ij\overline{k}}:=\eta^i\wedge\eta^j\wedge\overline{\eta}^k$.

Let us consider a generic left-invariant Hermitian metric $g$ on $(M,J)$, with associated fundamental form $\omega$ given by
\begin{equation*}
\omega=\frac{i}{2}\sum_{j=1}^3\alpha_{j\overline{j}}\eta^{j\overline{j}}+\frac{1}{2}\sum_{j<k}(\alpha_{j\overline{k}}-\overline{\alpha}_{j\overline{k}})\eta^{j\overline{k}},
\end{equation*}
with coefficients $\alpha_{j\overline{k}}\in\C$, for $j,k\in\{1,2,3\}$, such that the matrix representing $g$
\begin{equation*}
\begin{pmatrix}
\alpha_{1\overline{1}} & -i\alpha_{1\overline{2}} & -i\alpha_{1\overline{3}} \\
i\overline{\alpha}_{1\overline{2}} & \alpha_{2\overline{2}} & -i\alpha_{2\overline{3}} \\
i\overline{\alpha}_{1\overline{3}} & i\overline{\alpha}_{2\overline{3}} & \alpha_{3\overline{3}}
\end{pmatrix}
\end{equation*}
is positive definite.
From structure equations (\ref{Nak:struct_eq}), it is easy to check that $\delbar\omega^2=0$, hence any left-invariant Hermitian metric on $(M,J)$ is balanced.

We notice that the dimension of the space $H_{\delbar}^{0,1}(M)$ depends on the choice of the lattice $\Gamma=\left(\Z(a+ib)+\Z(c+id)\right)\ltimes_{\gamma}\Gamma''$, in particular on the choice of the real numbers $b$ and $d$. More accurately, it can be proved that, if $b,d\in2\pi\Z$, then $\dim H_{\delbar}^{0,1}(M)=3$, whereas, if either $b\notin 2\pi\Z$ or $d\notin 2\pi\Z$, then $\dim H_{\delbar}^{0,1}(M)=1$, see \cite{Kas}. Hence, we distinguish two cases.

\subsubsection{Case $(i)$: $b,d\in2\pi\Z$}
We define the following $\C$-base for $\mathcal{A}^{0,1}(M)$, consisting of the left-invariant $(0,1)$-forms $\{\tilde{\eta}^1,\tilde{\eta}^2,\tilde{\eta}^3\}$, defined as
\begin{align*}
\begin{cases}
\tilde{\eta}^1:=\overline{\eta}^1\\
\tilde{\eta}^2:=e^{\overline{z}^1-z^1}\overline{\eta}^2\\
\tilde{\eta}^3:=e^{z^1-\overline{z}^1}\overline{\eta}^3,
\end{cases}
\end{align*}
where the functions $e^{\overline{z}^1-z^1}$ and $e^{z^1-\overline{z}^1}$ are well defined on $M$ because of the choice of the lattice $\Gamma$.
%As shown in \cite{AK}, the Dolbeault cohomology of bidegree $(0,1)$ is isomorphic to the space $\C\langle \tilde{\phi}^1\rangle\otimes\C\langle\tilde{\phi}^2\rangle\otimes\C\langle\tilde{\phi}^3\rangle$

%The forms $\{\tilde{\eta}^1,\tilde{\eta}^2,\tilde{\eta}^3\}$ can be used to compute the Dolbeault cohomology of $(M,J)$, in particular, $H_{\delbar}^{0,1}(M)=\C\langle\tilde{\eta}^1\rangle\oplus \C\langle\tilde{\eta}^2\rangle\oplus\C\langle\tilde{\eta}^3\rangle$, and

Accordingly to \cite[Section 3]{Nak}, small deformations of $(M,J)$ can be characterized by means of the $(0,1)$-vector form
\begin{equation}
\phi(\mathbf{t})=\sum_{i,j=1}^3 t_{ij}\tilde{\eta}^j\otimes Z_i,
\end{equation}
with the coefficients of $\mathbf{t}=(t_{11},t_{12},t_{13},t_{21},t_{22},t_{23},t_{31},t_{32},t_{33})\in B(0,\delta)\subset\C^9$, $\delta>0$, belonging to one of the following classes:
\begin{align}
t_{11}\neq 0,t_{12}=t_{13}=t_{23}=t_{32}&=0;\label{Nak:Kur_1}\\
t_{11}=t_{22}=t_{33}&=0;\label{Nak:Kur_2}\\
t_{12}\neq 0, t_{11}=t_{13}=t_{21}=t_{23}=t_{31}&=0;\label{Nak:Kur_3}\\
t_{13}\neq0, t_{11}=t_{12}=t_{21}=t_{31}=t_{32}&=0.\label{Nak:Kur_4}
\end{align}
We can now make use of Theorem \ref{thm:main} and Corollary \ref{cor:main} to find obstruction for each class of small deformations of $(M,J)$.

\emph{Class} $(\ref{Nak:Kur_1})$.
In this case, the $(0,1)$-vector form parametrizing the deformation is
\begin{equation*}
\phi(\mathbf{t})=t_{11} \tilde{\eta}^1\otimes Z_1+t_{21}\tilde{\eta}^1\otimes Z_2 + t_{22}\tilde{\eta}^2\otimes Z_2 + t_{31}\tilde{\eta}^1\otimes Z_3 + t_{33}\tilde{\eta}^3\otimes Z_3, 
\end{equation*}
for $\mathbf{t}=(t_{11},t_{21},t_{22},t_{31},t_{33})\in B(0,\delta)\subset\C^5$, $\delta>0$.
We then consider the smooth curve of deformations
\begin{align*}
t\mapsto \phi(t):= t\,(a_{11} \tilde{\eta}^1\otimes Z_1+a_{21}\tilde{\eta}^1\otimes Z_2 + a_{22}\tilde{\eta}^2\otimes Z_2+ a_{31}\tilde{\eta}^1\otimes Z_3 + a_{33}\tilde{\eta}^3\otimes Z_3)\in\mathcal{A}^{0,1}(T^{1,0}(M))
\end{align*}
for $t\in I=(-\epsilon,\epsilon)$, $ \epsilon>0$, $(a_{11},a_{21},a_{22},a_{31},a_{33})\in\C^5$, whose derivative in $t=0$ is
\begin{equation*}
\phi'(0)=a_{11} \tilde{\eta}^1\otimes Z_1+a_{21}\tilde{\eta}^1\otimes Z_2 + a_{22}\tilde{\eta}^2\otimes Z_2 + a_{31}\tilde{\eta}^1\otimes Z_3 + a_{33}\tilde{\eta}^3\otimes Z_3.
\end{equation*}
With the aid of (\ref{Nak:struct_eq}), we compute
\begin{align*}
\del\circ i_{\phi'(0)}(\omega^2)&=[a_{12}(i\alpha_{1\overline{1}}\alpha_{2\overline{3}}+\overline{\alpha}_{1\c{2}}\alpha_{1\c{3}})+a_{32}(i\alpha_{3\overline{3}}\overline{\alpha}_{1\c{2}}-\c{\alpha}_{1\c{3}}\alpha_{2\c{3}})]e^{\c{z}^1-z^1}\eta^{12}\wedge\eta^{\tilde{1}\tilde{2}\tilde{3}}\\
&+\frac{1}{2}\left[a_{11}(i\alpha_{2\c{2}}\alpha_{1\c{3}}-\alpha_{1\c{2}}\alpha_{2\c{3}})+a_{31}(|\alpha_{2\c{3}}|^2-\alpha_{2\c{2}}\alpha_{3\c{3}})\right]\eta^{12}\wedge\eta^{\tilde{1}\tilde{2}\tilde{3}}\\
&+[a_{13}(\alpha_{1\c{2}}\c{\alpha}_{1\c{3}}-i\alpha_{1\c{3}}\c{\alpha}_{2\c{3}})+a_{23}(i\alpha_{2\c{2}}\c{\alpha}_{1\c{3}}+\c{\alpha}_{1\c{2}}\c{\alpha}_{2\c{3}})]e^{z^1-\c{z}^1}\eta^{13}\wedge\eta^{\tilde{1}\tilde{2}\tilde{3}}\\
&+\frac{1}{2}[a_{11}(i\alpha_{1\c{1}}\alpha_{1\c{2}}+\alpha_{1\c{3}}\c{\alpha}_{2\c{3}})+a_{21}(|\alpha_{2\c{3}}|^2-\alpha_{2\c{2}}\alpha_{3\c{3}})]\eta^{13}\wedge\eta^{\tilde{1}\tilde{2}\tilde{3}}.
\end{align*}
We note that the forms $e^{\c{z}^1-z^1}\eta^{12}\wedge\eta^{\tilde{1}\tilde{2}\tilde{3}}$  and $e^{z^1-\c{z}^1}\eta^{13}\wedge\eta^{\tilde{1}\tilde{2}\tilde{3}}$ are $\delbar$-exact. In fact,
\begin{align*}
&e^{\c{z}^1-z^1}\eta^{12}\wedge\eta^{\tilde{1}\tilde{2}\tilde{3}}=\delbar(e^{\c{z}^1-z^1}\eta^{12}\wedge\tilde{\eta}^{23})\\
&e^{z^1-\c{z}^1}\eta^{13}\wedge\eta^{\tilde{1}\tilde{2}\tilde{3}}=\delbar(-e^{z^1-\c{z}^1}\eta^{13}\wedge\tilde{\eta}^{23}),
\end{align*}
therefore they both represent a vanishing class in $H_{\delbar}^{2,3}(M)$. On the other hand, it can be easily shown that the forms $\eta^{12}\wedge\eta^{\tilde{1}\tilde{2}\tilde{3}}$ and $\eta^{13}\wedge\eta^{\tilde{1}\tilde{2}\tilde{3}}$ are harmonic with respect to the Dolbeault Laplacian operator, i.e., they belong to $\mathcal{H}_{\delbar}^{2,3}(M,g)$. As a consequence, they correspond, respectively, to non-vanishing cohomology classes $[\eta^{12}\wedge\eta^{\tilde{1}\tilde{2}\tilde{3}}]_{\delbar}$ and $[\eta^{13}\wedge\eta^{\tilde{1}\tilde{2}\tilde{3}}]_{\delbar}$ in $H_{\delbar}^{2,3}(X).$ Hence, by Corollary \ref{cor:main}, if one of the following equations does not hold
\begin{align*}
\begin{cases}
a_{11}(i\alpha_{1\c{1}}\alpha_{1\c{2}}+\alpha_{1\c{3}}\c{\alpha}_{2\c{3}})+a_{21}(|\alpha_{2\c{3}}|^2-\alpha_{2\c{2}}\alpha_{3\c{3}})=0\\
a_{11}(i\alpha_{2\c{2}}\alpha_{1\c{3}}-\alpha_{1\c{2}}\alpha_{2\c{3}})+a_{31}(|\alpha_{2\c{3}}|^2-\alpha_{2\c{2}}\alpha_{3\c{3}})=0,
\end{cases}
\end{align*}
there exists no curve of balanced metrics $\{\omega_t\}_{t\in I}$ such that $\omega_0=\omega$ along the curve of deformations $t\mapsto\phi(t)$.

\emph{Class} $(\ref{Nak:Kur_2})$.
The deformation is parametrized by the $(0,1)$-vector form $\phi(\mathbf{t})$, with 
\begin{equation*}
\phi(\mathbf{t})=t_{12}\tilde{\eta}^2\otimes Z_1+t_{13}\tilde{\eta}^3\otimes Z_1+t_{21}\tilde{\eta}^1\otimes Z_2+t_{23}\tilde{\eta}^3\otimes Z_2+t_{31}\tilde{\eta}^1\otimes Z_3+t_{32}\tilde{\eta}^2\otimes Z_3,
\end{equation*}
with $\mathbf{t}=(t_{12},t_{13},t_{21},t_{23},t_{31},t_{32})\in B(0,\delta)\subset \C^6$, $\delta>0$.

We consider the smooth curve of deformations
\begin{align*}
t\mapsto \phi(t):=&t(a_{12}\tilde{\eta}^2\otimes Z_1+a_{13}\tilde{\eta}^3\otimes Z_1+a_{21}\tilde{\eta}^1\otimes Z_2\\
&+a_{23}\tilde{\eta}^3\otimes Z_2+a_{31}\tilde{\eta}^1\otimes Z_3+a_{32}\tilde{\eta}^2\otimes Z_3),
\end{align*}
for $t\in I=(-\epsilon,\epsilon), \epsilon>0$, whose derivative in $t=0$ is
\begin{align*}
\phi'(0)=&a_{12}\tilde{\eta}^2\otimes Z_1+a_{13}\tilde{\eta}^3\otimes Z_1+a_{21}\tilde{\eta}^1\otimes Z_2\\
&+a_{23}\tilde{\eta}^3\otimes Z_2+a_{31}\tilde{\eta}^1\otimes Z_3+a_{32}\tilde{\eta}^2\otimes Z_3.
\end{align*}
With the aid of (\ref{Nak:struct_eq}), we compute
\begin{align*}
\del\circ i_{\phi'(0)}(\omega^2)&=[a_{12}(i\alpha_{1\overline{1}}\alpha_{2\overline{3}}+\overline{\alpha}_{1\c{2}}\alpha_{1\c{3}})+a_{32}(i\alpha_{3\overline{3}}\overline{\alpha}_{1\c{2}}-\c{\alpha}_{1\c{3}}\alpha_{2\c{3}})]e^{\c{z}^1-z^1}\eta^{12}\wedge\eta^{\tilde{1}\tilde{2}\tilde{3}}\\
&+\frac{1}{2}\left[a_{31}(|\alpha_{2\c{3}}|^2-\alpha_{2\c{2}}\alpha_{3\c{3}})\right]\eta^{12}\wedge\eta^{\tilde{1}\tilde{2}\tilde{3}}\\
&+[a_{13}(\alpha_{1\c{2}}\c{\alpha}_{1\c{3}}-i\alpha_{1\c{3}}\c{\alpha}_{2\c{3}})+a_{23}(i\alpha_{2\c{2}}\c{\alpha}_{1\c{3}}+\c{\alpha}_{1\c{2}}\c{\alpha}_{2\c{3}})]e^{z^1-\c{z}^1}\eta^{13}\wedge\eta^{\tilde{1}\tilde{2}\tilde{3}}\\
&+\frac{1}{2}[a_{21}(|\alpha_{2\c{3}}|^2-\alpha_{2\c{2}}\alpha_{3\c{3}})]\eta^{13}\wedge\eta^{\tilde{1}\tilde{2}\tilde{3}}.
\end{align*}
We observe that, again, since the forms $e^{\c{z}^1-z^1}\eta^{12}\wedge\eta^{\tilde{1}\tilde{2}\tilde{3}}$ and $e^{z^1-\c{z}^1}\eta^{13}\wedge\eta^{\tilde{1}\tilde{2}\tilde{3}}$ are cohomologous to $0$ in $H_{\delbar}^{2,3}(M)$ and the forms $\eta^{12}\wedge\eta^{\tilde{1}\tilde{2}\tilde{3}}$ and $\eta^{13}\wedge\eta^{\tilde{1}\tilde{2}\tilde{3}}$ are $\delbar$-harmonic, the obstruction from Corollary \ref{cor:main} boils down to
\begin{align*}
&a_{21}(|\alpha_{2\c{3}}|^2-\alpha_{2\c{2}}\alpha_{3\c{3}})=0\\
&a_{31}(|\alpha_{2\c{3}}|^2-\alpha_{2\c{2}}\alpha_{3\c{3}})=0.
\end{align*}
We point out that, since the metric $g$ is Hermitian and, hence, positive definite, the real number $|\alpha_{2\c{3}}|^2-\alpha_{2\c{2}}\alpha_{3\c{3}}$ is strictly positive.
Therefore, there exists no curve of balanced metrics $\{\omega_t\}_{t\in I}$ such that $\omega_0=\omega$ along the curve of deformations $t\mapsto\phi(t)$, if
\begin{align*}
\begin{pmatrix}
a_{21}\\
a_{31}
\end{pmatrix}\neq
\begin{pmatrix}
0\\
0
\end{pmatrix}.
\end{align*}

\emph{Class} $(\ref{Nak:Kur_3})$.
For this class, the $(0,1)$-vector deformation form is
\begin{equation*}
\phi(\mathbf{t})=t_{12}\tilde{\eta}^2\otimes Z_1+t_{22}\tilde{\eta}^2\otimes Z_2+t_{32}\tilde{\eta}^2\otimes Z_3+t_{33}\tilde{\eta}^3\otimes Z_3,
\end{equation*}
for $\mathbf{t}=(t_{12},t_{22},t_{32},t_{33})\in B(0,\delta)\subset\C^4$, $\delta>0.$ We consider the smooth curve of deformations
\begin{align*}
t\mapsto\phi(t):=t(a_{12}\tilde{\eta}^2\otimes Z_1+a_{22}\tilde{\eta}^2\otimes Z_2+a_{32}\tilde{\eta}^2\otimes Z_3+a_{33}\tilde{\eta}^3\otimes Z_3),
\end{align*}
for $t\in I=(-\epsilon,\epsilon),\epsilon>0$, whose derivative in $t=0$ is
\begin{align*}
\phi'(0)=a_{12}\tilde{\eta}^2\otimes Z_1+a_{22}\tilde{\eta}^2\otimes Z_2+a_{32}\tilde{\eta}^2\otimes Z_3+a_{33}\tilde{\eta}^3\otimes Z_3.
\end{align*}
In this case, $\del\circ i_{\phi'(0)}(\omega^2)=0$, therefore Corollary \ref{cor:main} gives no obstruction to the existence of smooth curves of balanced metrics $\{\omega_t\}_{t\in I}$ such that $\omega_0=\omega$ along the curve of deformations $t\mapsto \phi(t)$.
Moreover, if $\{\omega_t\}_{t\in I}$ is any smooth curve of left invariant Hermitian metrics along $\phi(t)$ such that $\omega_0=\omega$, we can see that $\delbar(\omega^2(0))'=0$, where we have set $\omega_t=e^{i_{\phi(t)}|i_{\overline{\phi(t)}}}\omega(t)$, for $\omega(t)=\omega_{ij}(t)dz^i\wedge d\overline{z}^j\in\mathcal{A}^{1,1}(M)$. Therefore, also Theorem \ref{thm:main} yields no obstruction.

\emph{Class} $(\ref{Nak:Kur_4})$.
The $(0,1)$-vector form for this class is
\begin{equation*}
\phi(\mathbf{t})=t_{13}\tilde{\eta}^3\otimes Z_1+t_{22}\tilde{\eta}^2\otimes Z_2+t_{23}\tilde{\eta}^3\otimes Z_2+t_{33}\tilde{\eta}^3\otimes Z_3,
\end{equation*}
for $\mathbf{t}=(t_{13},t_{22},t_{23},t_{33})\in B(0,\delta)\subset\C^4$, $\delta>0$.

Let us consider the smooth curve of deformations
\begin{align*}
t\mapsto \phi(t):=t(a_{13}\tilde{\eta}^3\otimes Z_1+a_{22}\tilde{\eta}^2\otimes Z_2+a_{23}\tilde{\eta}^3\otimes Z_2+a_{33}\tilde{\eta}^3\otimes Z_3)
\end{align*}
for $t\in(-\epsilon,\epsilon)$ and its derivative in $t=0$
\begin{align*}
\phi'(0)=a_{13}\tilde{\eta}^3\otimes Z_1+a_{22}\tilde{\eta}^2\otimes Z_2+a_{23}\tilde{\eta}^3\otimes Z_2+a_{33}\tilde{\eta}^3\otimes Z_3.
\end{align*}
Also in this case, $\del\circ i_{\phi'(0)}(\omega^2)=0$, i.e., Corollary \ref{cor:main} yields no obstruction and analogously to the previous class, also Theorem \ref{thm:main} yields no non-trivial conditions.

We can focus now on the other case.

\subsubsection{Case (ii): $c\notin 2\pi\Z$ or $d\notin 2\pi\Z$}
%Let us consider the Lie group $G:=\C\ltimes_{\gamma}\C^2$, with $\gamma\colon \C\rightarrow GL(\C^2)$ defined by
%\[
%\gamma(x+iy):=\begin{pmatrix}
%e^x & 0\\
%0 & e^{-x}
%\end{pmatrix}, \quad \forall x+iy\in\C.
%\]
%Then, for some $a\in\R$, the matrix $\begin{pmatrix}
%e^a & 0\\
%0 & e^{-a}
%\end{pmatrix}$
%is conjugate to an element of $SL(2;\Z)$. Then, for any $ 0\neq b\in\R$, we have a lattice $\Gamma=(a\Z+ib\Z)\ltimes\Gamma''$ of $G$, with $\Gamma''$ a lattice of $\C^2$. The compact quotient $M:=\Gamma/G$ is called the \emph{completely-solvable Nakamura Manifold}, see \cite[Section 2]{Nak} for further details on the construction.

In \cite[Section 3]{Nak}, it is shown that $H_{\delbar}^{0,1}(M)=\C\langle\overline{\eta}^1\rangle$, and any small deformation of $(M,J)$ can be parametrized by the $(0,1)$-vector form
\begin{align*}
\phi(\mathbf{t}):=t_1\overline{\eta}^1\otimes Z_1+t_2\overline{\eta}^1\otimes Z_2+t_3\overline{\eta}^1\otimes Z_3,
\end{align*}
with $\mathbf{t}=(t_1,t_2,t_3)\in B(0,\delta)\subset\C^3$, $\delta>0$.
We can then consider the smooth curve of deformations
\begin{align*}
t\mapsto \phi(t):=t(a_1\overline{\eta}^1\otimes Z_1+a_2\overline{\eta}^1\otimes Z_2+a_3\overline{\eta}^1\otimes Z_3),
\end{align*}
for $t\in(-\epsilon,\epsilon)$, $\epsilon>0$, $(a_1,a_2,a_3)\in\C^3$, whose derivative in $t=0$ is
\begin{align*}
\phi'(0)=a_1\overline{\eta}^1\otimes Z_1+a_2\overline{\eta}^1\otimes Z_2+a_3\overline{\eta}^1\otimes Z_3.
\end{align*}

By making use of (\ref{Nak:struct_eq}), we  compute
\begin{align*}
\del\circ i_{\phi'(0)}(\omega^2)=\frac{1}{2}(a_2(|\alpha_{2\c{3}}|^2-\alpha_{2\c{2}}\alpha_{3\c{3}})+a_1(i\alpha_{3\c{3}}\alpha_{1\c{2}}+\alpha_{1\c{3}}\c{\alpha}_{2\c{3}}))\eta^{13\overline{123}}\\
+\frac{1}{2}(a_3(|\alpha_{2\c{3}}|^2-\alpha_{2\c{2}}\alpha_{3\c{3}})+a_1(i\alpha_{2\c{2}}\alpha_{1\c{3}}-\alpha_{1\c{2}}\alpha_{2\c{3}}))\eta^{12\c{123}}.
\end{align*}
We can easily verify that $\delbar\eta^{12\c{123}}=\delbar^*\eta^{12\c{123}}=\delbar\eta^{13\c{123}}=\delbar^*\eta^{13\c{123}}=0$, i.e., the $(2,3)$-forms $\eta^{12\c{123}}$ and $\eta^{13\overline{123}}$ are $\delbar-$harmonic. Therefore, the Dolbeault cohomology classes $[\eta^{12\c{123}}]_{H_{\delbar}^{2,3}(M)}$ and $[\eta^{13\overline{123}}]_{H_{\delbar}^{2,3}(M)}$ are not vanishing. On this accounts, Corollary \ref{cor:main} implies that if there exists a smooth curve of balanced metrics $\{\omega_t\}_{t\in I}$ along the smooth curve of deformations $t\mapsto \phi(t)$, then we must have that
\begin{align}
\begin{cases}\label{Nakcs_cond}
a_2(|\alpha_{2\c{3}}|^2-\alpha_{2\c{2}}\alpha_{3\c{3}})+a_1(i\alpha_{3\c{3}}\alpha_{1\c{2}}+\alpha_{1\c{3}}\c{\alpha}_{2\c{3}})=0\\
a_3(|\alpha_{2\c{3}}|^2-\alpha_{2\c{2}}\alpha_{3\c{3}})+a_1(i\alpha_{2\c{2}}\alpha_{1\c{3}}-\alpha_{1\c{2}}\alpha_{2\c{3}})=0.
\end{cases}
\end{align}
We notice that, if $a_1=0$, i.e., $\phi'(0)=a_2\overline{\eta}^1\otimes Z_2+a_3\overline{\eta}^1\otimes Z_3$, condition (\ref{Nakcs_cond}) becomes
\begin{align*}
\begin{cases}
a_2=0 \\
a_3=0\\
\end{cases}
\end{align*}
since $|\alpha_{2\c{3}}|^2-\alpha_{2\c{2}}\alpha_{3\c{3}}\neq 0$, being $g$ a Hermitian metric. Hence, by Corollary \ref{cor:main}, we can conclude that there exists no smooth curve of balanced metrics $\{\omega_t\}_{t\in I}$ such that $\omega_0=\omega$, along $\phi(t)$   with $\phi'(0)=a_2\overline{\eta}^1\otimes Z_2+a_3\overline{\eta}^1\otimes Z_3$. 

Viceversa, let us consider the case in which $a_1\neq 0$ and at least one between $a_2$ and $a_3$ vanishes, i.e., for example, $a_2=0$. Then, condition (\ref{Nakcs_cond}) reduces to
\begin{align*}
\begin{cases}
a_1=0\\
a_3(|\alpha_{2\c{3}}|^2-\alpha_{2\c{2}}\alpha_{3\c{3}})+a_1(i\alpha_{2\c{2}}\alpha_{1\c{3}}-\alpha_{1\c{2}}\alpha_{2\c{3}})=0,
\end{cases}
\end{align*}
since the term $i\alpha_{2\c{2}}\alpha_{1\c{3}}-\alpha_{1\c{2}}\alpha_{2\c{3}}\neq 0$, being $g$ a Hermitian metric. We assumed $a_1\neq 0$, therefore by Corollary \ref{cor:main}, there exists no smooth curve of balanced metrics $\{\omega_t\}_{t\in I}$ such that $\omega_0=\omega$, along the smooth curve of deformations $\phi(t)$ with $\phi'(0)=a_1\overline{\eta}^1\otimes Z_1+a_3\overline{\eta}^1\otimes Z_3$. We come to the same conclusion if we consider $a_3=0$.

We can then summarize what we obtained in the following theorems.
\begin{thm}
Let $(M,J)$ be the complex parallelisable Nakamura manifold with $\dim H_{\delbar}^{0,1}(M)=3$, where $J$ is the integrable left-invariant almost-complex structure induced by the left-invariant coframe $\{\eta^1,\eta^2,\eta^3\}$ with structure equations
\begin{align*}
\begin{cases}
d\eta^1=0\\
d\eta^2=-\eta^{12}\\
d\eta^3=\eta^{13}.
\end{cases}
\end{align*}
Let $g$ be any left-invariant Hermitian (balanced) metric with associated fundamental form
\[
\omega=\frac{i}{2}\sum_{j=1}^3\alpha_{j\overline{j}}\eta^{j\overline{j}}+\frac{1}{2}\sum_{j<k}(\alpha_{j\overline{k}}-\overline{\alpha}_{j\overline{k}})\eta^{j\overline{k}}.
\]
Defining the left-invariant $(0,1)$-forms $\{\tilde{\eta}^1,\tilde{\eta}^2,\tilde{\eta}^3\}$ by
\begin{align*}
&\tilde{\eta}^1:=\overline{\eta}^1\\
&\tilde{\eta}^2:=e^{\c{z}^1-z^1}\overline{\eta}^2\\
&\tilde{\eta}^3:=e^{z^1-\c{z}^1}\overline{\eta}^3,
\end{align*}
let $t\mapsto\phi(t):=t\sum_{i,j=1}^3a_{ij}\tilde{\eta}^j\otimes Z_i\in\mathcal{A}^{0,1}(T^{1,0}(M))$ be a smooth curve of deformations of $(M,J)$, for $\{a_{ij}\}_{i,j=1}^3\subset\C$, $t\in I=(-\epsilon,\epsilon)$, $\epsilon>0$.

Then, 
\begin{itemize}
\item if $a_{11}\neq0,a_{12}=a_{13}=a_{23}=a_{32}=0$, there exists no smooth curve of balanced metrics $\{\omega_t\}_{t\in I}$ such that $\omega_0=\omega$, along the curve of deformation $t\mapsto\phi(t)$, if
\begin{align*}
\begin{pmatrix}
a_{11}(i\alpha_{2\c{2}}\alpha_{1\c{3}}-\alpha_{1\c{2}}\alpha_{2\c{3}})+a_{31}(|\alpha_{2\c{3}}|^2-\alpha_{2\c{2}}\alpha_{3\c{3}})=0\\
a_{11}(i\alpha_{1\c{1}}\alpha_{1\c{2}}+\alpha_{1\c{3}}\c{\alpha}_{2\c{3}})+a_{21}(|\alpha_{2\c{3}}|^2-\alpha_{2\c{2}}\alpha_{3\c{3}})=0
\end{pmatrix}\neq
\begin{pmatrix}
0\\
0
\end{pmatrix};
\end{align*}
\item if $a_{11}=a_{22}=a_{33}=0$, there exists no smooth curve of balanced metrics $\{\omega_t\}_{t\in I}$ such that $\omega_0=\omega$, along the curve of deformation $t\mapsto\phi(t)$, if
\begin{align*}
\begin{pmatrix}
a_{21}\\
a_{31}
\end{pmatrix}\neq
\begin{pmatrix}
0\\
0
\end{pmatrix}.
\end{align*}
\end{itemize}
\end{thm}

\begin{thm}
Let $(M,J)$ be the complex parallelisable Nakamura manifold with $\dim H_{\del}^{0,1}(M)=1$, where $J$ is the integrable left-invariant almost-complex structure induced by the left-invariant coframe $\{\eta^1,\eta^2,\eta^3\}$ with structure equations
\begin{align*}
\begin{cases}
d\eta^1=0\\
d\eta^2=-\eta^{12}\\
d\eta^3=\eta^{13}.
\end{cases}
\end{align*}
Let $g$ be any left-invariant Hermitian (balanced) metric with associated fundamental form
\[
\omega=\frac{i}{2}\sum_{j=1}^3\alpha_{j\overline{j}}\eta^{j\overline{j}}+\frac{1}{2}\sum_{j<k}(\alpha_{j\overline{k}}-\overline{\alpha}_{j\overline{k}})\eta^{j\overline{k}}.
\]
Let $t\mapsto\phi(t):=t\sum_{i}^3a_{i}\overline{\eta}^1\otimes Z_i\in\mathcal{A}^{0,1}(T^{1,0}(M))$ be a smooth curve of deformations of $(M,J)$, for $0\neq(a_1,a_2,a_3)\in\C^3$, $t\in I=(-\epsilon,\epsilon)$, $\epsilon>0$.

Then, there exists no smooth curve of balanced metrics $\{\omega_t\}_{t\in I}$ such that $\omega_0=\omega$, along the curve of deformation $t\mapsto\phi(t)$, if
\begin{align*}
\begin{pmatrix}
a_2(|\alpha_{2\c{3}}|^2-\alpha_{2\c{2}}\alpha_{3\c{3}})+a_1(i\alpha_{3\c{3}}\alpha_{1\c{2}}+\alpha_{1\c{3}}\c{\alpha}_{2\c{3}})\\
a_3(|\alpha_{2\c{3}}|^2-\alpha_{2\c{2}}\alpha_{3\c{3}})+a_1(i\alpha_{2\c{2}}\alpha_{1\c{3}}-\alpha_{1\c{2}}\alpha_{2\c{3}})
\end{pmatrix}\neq
\begin{pmatrix}
0\\
0
\end{pmatrix}.
\end{align*}

In particular, if one the following holds:
\begin{itemize}
\item $a_1=0$;
\item $a_1\neq 0$, $(a_2,a_3)\in\{(a_2,0),(0,a_3)\}$,
\end{itemize}
there exists no smooth curve of balanced metrics $\{\omega_t\}_{t\in I}$ such that $\omega_0=\omega$, along the curve of deformation $t\mapsto\phi(t)$.

\end{thm}

\subsection{Example 2}{\em (Iwasawa manifold).}
Let $G=H(3;\C)$ be the $3$-dimensional complex Heisenberg group.
%\begin{align*}
%G=\left\{ \begin{pmatrix}
%1 & z^1 & z^3\\
%0 & 1 & z^2\\
%0 & 0 & 1
%\end{pmatrix}
%: z^1,z^2,z^3\in\C
%\right\}.
%\end{align*}
It well known that $G$ is a $2$-step nilpotent Lie group. Let us consider the lattice $\Gamma:=H(3,\Z[i])$ of $G$, i.e., $\Gamma=H(3;\C)\cap \GL(3;\Z[i])$. The quotient $M:=\Gamma/G$ is a compact manifold, known as the \emph{Iwasawa manifold}. In particular, $M$ is a $3$-dimensional $2$-step complex nilmanifold with universal covering $\C^3$.

If $\{z^1,z^2,z^3\}$ are the standard coordinates on $\C^3$, the forms $\{\eta^1,\eta^2,\eta^3\}$, defined by
\begin{align*}
\begin{cases}
\eta^1=dz^1\\
\eta^2=dz^2\\
\eta^3=dz^3-z^1dz^2,
\end{cases}
\end{align*}
are a left-invariant coframe of $(1,0)$-forms on $G$, therefore they descend to the quotient $M$. The dual frame of $(1,0)$-vector fields $\{Z_1,Z_2,Z_3\}$ on $G$ has local expression
\begin{align*}
\begin{cases}
Z_1=\frac{\del}{\del z^1}+z^1\frac{\del}{\del z^3}\\
Z_2=\frac{\del}{\del z^2}\\
Z_3=\frac{\del}{\del z^3}.
\end{cases}
\end{align*}
We notice that, by looking at structure equations
\begin{align}\label{Iwa:struct_eq}
\begin{cases}
d\eta^1=0\\
d\eta^2=0\\
d\eta^3=-\eta^{12},
\end{cases}
\end{align}
the coframe $\{\eta^1,\eta^2,\eta^3\}$ induces a left invariant almost-complex structure $J$ on $M$, which is integrable.

Let $g$ be any left-invariant Hermitian metric on $(M,J)$. Its associated fundamental form $\omega$ can be written as
\begin{align*}
\omega=\frac{i}{2}\sum_{j=1}^3\alpha_{j\overline{j}}\eta^{j\overline{j}}+\frac{1}{2}\sum_{j<k}(\alpha_{j\overline{k}}-\overline{\alpha}_{j\overline{k}})\eta^{j\overline{k}},
\end{align*}
with complex numbers $\{\alpha_{j\overline{k}}\}_{j,k=1}^3$ such that the matrix respresenting $g$
\begin{equation*}
\begin{pmatrix}
\alpha_{1\overline{1}} & -i\alpha_{1\overline{2}} & -i\alpha_{1\overline{3}} \\
i\overline{\alpha}_{1\overline{2}} & \alpha_{2\overline{2}} & -i\alpha_{2\overline{3}} \\
i\overline{\alpha}_{1\overline{3}} & i\overline{\alpha}_{2\overline{3}} & \alpha_{3\overline{3}}
\end{pmatrix}
\end{equation*}
is positive definite. By structure equations (\ref{Iwa:struct_eq}), it is easy to check that $\delbar\omega^2=0$, i.e., the left-invariant Hermitian metric $g$ is balanced. 

In \cite{Nak}, Nakamura gives a complete description of Kuranishi space of the Iwasawa manifold. In particular, any small deformation of $(M,J)$ can be parametrized by the $(0,1)$-vector form
\begin{align*}
\phi(\mathbf{t})=\sum_{i=1}^3\sum_{j=1}^2t_{ij}\overline{\eta}^j\otimes Z_i - (t_{11}t_{22}-t_{12}t_{21})\overline{\eta}^3\otimes Z_3,
\end{align*}
with $\mathbf{t}=(t_{11},t_{12},t_{21},t_{22},t_{31},t_{32})\in B(0,\delta)\subset\C^6$, $\delta>0$.

Let us consider the smooth curve of deformations
\begin{align*}
t\mapsto \phi(t):=&t(a_{11}\overline{\eta}^1\otimes Z_1+a_{12}\overline{\eta}^2\otimes Z_1+a_{21}\overline{\eta}^1\otimes Z_2+a_{22}\overline{\eta}^2\otimes Z_2+a_{31}\overline{\eta}^1\otimes Z_3\\
&+a_{32}\overline{\eta}^2\otimes Z_3)
-t^2(a_{11}a_{22}-a_{12}a_{21})\overline{\eta}^3\otimes Z_3\in\mathcal{A}^{0,1}(T^{1,0}(M)),
\end{align*}
with $t\in I=(-\epsilon,\epsilon)$, $\epsilon>0$ and $(a_{11},a_{12},a_{21},a_{22},a_{31},a_{32})\in\C^6$. Its derivative in $t=0$ is
\begin{align*}
\phi'(0)=a_{11}\overline{\eta}^1\otimes Z_1+a_{12}\overline{\eta}^2\otimes Z_1+a_{21}\overline{\eta}^1\otimes Z_2+a_{22}\overline{\eta}^2\otimes Z_2+a_{31}\overline{\eta}^1\otimes Z_3a_{32}\overline{\eta}^2\otimes Z_3.
\end{align*}
With the aid of structure equations (\ref{Iwa:struct_eq}), we compute
\begin{align*}
\del\circ i_{\phi'(0)}(\omega^2)=&\frac{1}{2}\Big(a_{12}(|\alpha_{1\c{3}}|^2-\alpha_{1\c{1}}\alpha_{3\c{3}})+a_{21}(\alpha_{2\c{2}}\alpha_{3\c{3}}-|\alpha_{2\c{3}}|^2)\\
&-a_{11}(i\alpha_{3\c{3}}\alpha_{1\c{2}}+\alpha_{1\c{3}}\overline{\alpha}_{2\c{3}})+a_{22}(-i\alpha_{3\c{3}}\c{\alpha}_{1\c{3}}+\c{\alpha}_{1\c{3}}\alpha_{2\c{3}})\Big)\eta^{12\c{123}}.
\end{align*}
We notice that the $(2,3)$-form $\eta^{12\c{123}}$ is both $\delbar$-closed and $\delbar^*$-closed, i.e., it is $\delbar$-harmonic. Hence, the corresponding Dolbeault class $[\eta^{12\c{123}}]_{\delbar}$ is non-vanishing in $H_{\delbar}^{2,3}(M)$.
Applying Corollary \ref{cor:main}, we see that there exists no smooth curve of balanced metrics $\{\omega_t\}_{t\in I}$ along the curve of deformations $t\mapsto\phi(t)$, such that $\omega_0=\omega$, if the following equation holds
\begin{align*}
a_{12}(|\alpha_{1\c{3}}|^2-\alpha_{1\c{1}}\alpha_{3\c{3}})+a_{21}(\alpha_{2\c{2}}\alpha_{3\c{3}}-|\alpha_{2\c{3}}|^2)-a_{11}(i\alpha_{3\c{3}}\alpha_{1\c{2}}+\alpha_{1\c{3}}\overline{\alpha}_{2\c{3}})+a_{22}(-i\alpha_{3\c{3}}\c{\alpha}_{1\c{3}}+\c{\alpha}_{1\c{3}}\alpha_{2\c{3}})\neq 0.
\end{align*}
We observe that, for $a_{ij}=0$ for $(i,j)\neq (1,2)$, we find the same curve of deformations that Alessandrini and Bassanelli costructed in \cite{AB90} to prove the non stability of the balanced condition under small deformations of the complex structure.

We gather what we have obtained in the following theorem.
\begin{thm}
Let $(M,J)$ be the Iwasawa manifold with integrable left-invariant complex structure $J$, induced by the left-invariant coframe $\{\eta^1,\eta^2,\eta^3\}$ with structure equations
\begin{align*}
\begin{cases}
d\eta^1=0\\
d\eta^2=0\\
d\eta^3=-\eta^{12}.
\end{cases}
\end{align*}
Let $g$ be a left-invariant Hermitian (balanced) metric on $(M,J)$ with associated fundamental form
\begin{align*}
\omega=\frac{i}{2}\sum_{j=1}^3\alpha_{j\overline{j}}\eta^{j\overline{j}}+\frac{1}{2}\sum_{j<k}(\alpha_{j\overline{k}}-\overline{\alpha}_{j\overline{k}})\eta^{j\overline{k}}.
\end{align*}
Let $t\mapsto\phi(t)=t(\sum_{i=1}^3\sum_{j=1}^2a_{ij}\overline{\eta}^j\otimes Z_j)-t^2(a_{11}a_{22}-a_{12}a_{21})\overline{\eta}^3\otimes Z_3\in\mathcal{A}^{0,1}(T^{1,0}(M))$ be a smooth curve of deformations of $(M,J)$, with $\tensor*{\{a_{ij}\}}{_{i=1}^3_{j=1}^{\,\,\,\,\,\,2}}\subset\C$, $t\in I=(-\epsilon,\epsilon)$, $\epsilon>0$.

Then, if the following condition holds
\begin{align*}
a_{12}(|\alpha_{1\c{3}}|^2-\alpha_{1\c{1}}\alpha_{3\c{3}})+a_{21}(\alpha_{2\c{2}}\alpha_{3\c{3}}-|\alpha_{2\c{3}}|^2)-a_{11}(i\alpha_{3\c{3}}\alpha_{1\c{2}}+\alpha_{1\c{3}}\overline{\alpha}_{2\c{3}})+a_{22}(-i\alpha_{3\c{3}}\c{\alpha}_{1\c{3}}+\c{\alpha}_{1\c{3}}\alpha_{2\c{3}})\neq 0,
\end{align*}
there exists no smooth curve of balanced metrics $\{\omega_t\}_{t\in I}$ such that $\omega_0=\omega$ along the curve of deformations $t\mapsto\phi(t)$.
\end{thm}


\begin{thebibliography}{20}

\bibitem{AB90} L. Alessandrini, G. Bassanelli, Small deformations of a class of compact non-K\"ahler manifolds, {\em Proc. Amer. Math. Soc.} {\bf 109} (1990), n. 4, 1059–1062.


\bibitem{AB91} L. Alessandrini, G. Bassanelli, Compact p-Kähler manifolds, \emph{Geom. Dedicata} {\bf 38} (1991), n. 2, 199–210.


\bibitem{AU} D. Angella, L. Ugarte, On small deformations of balanced manifolds, {\em Differential Geom. Appl.} {\bf 54} (2017), 464-474.
%\bibitem{AT} D. Angella, A. Tomassini, On cohomological decomposition of almost-complex manifolds and deformations, {\em J. Symplectic Geom.}, {\bf 9} (2011), no. 3, 403–428.

%\bibitem{B} P. de Bartolomeis, F. Meylan, Intrinsic deformation theory of CR structures, {\em Ann. Sc. Norm. Super. Pisa Cl. Sci.}, {\bf 9} (2010), n. 3, 459–494.

%\bibitem{CFP06} S. Console, A. Fino, Y.S. Poon, Stability of Abelian complex structures, \emph{Internat. J. Math.} {\bf 17} (2006), n. 4, 401-416.

%\bibitem{E} N. Egidi, Special metrics on compact complex manifolds, \emph{Diff. Geom. Appl.} {\bf 14} (2001), 217–234.

\bibitem{E47} C. Ehresmann, Sur les espaces fibres differentiables, \emph{C. R. Acad. Sci. Paris} {\bf 224} (1947), 1611–1612.

%\bibitem{FG} A. Fino and G. Grantcharov, Properties of manifolds with skewsymmetric torsion and special holonomy, \emph{Adv. Math.}  {\bf 189} (2004), 439–450.

%\bibitem{FPS04} A. Fino, M. Parton, S. Salamon, Families of strong KT structures in six dimensions, {\em Comment. Math. Helv.} {\bf 79} (2004), 317–340.

%\bibitem{FT09} A. Fino, A. Tomassini, Blow-ups and resolutions of strong K\"ahler with torsion metrics, {\em Adv. Math.}  {\bf 221} (2009), n. 3, 914–935.

%\bibitem{FT11} A. Fino, A. Tomassini, On Astheno-K\"ahler metrics, {\em J. London. Math. Soc.} {\bf 83} (2011), n. 2, 290-308.

\bibitem{FY} J.X. Fu, S.T. Yau, The theory of superstring with flux on non-Kähler manifolds and the complex Monge-Ampère equation, {\em J. Differential Geom.} {\bf 78} (2008), n. 3,  369-428.


%\bibitem{Ga} P. Gauduchon, Le théorème de l'excentricité nulle. (French) {\em C. R. Acad. Sci. Paris Sér. A-B} {\bf 285} (1977), n. 5, 387-390.

%\bibitem{Ga97} P. Gauduchon, Hermitian connnections and Dirac operators,{\em Boll. Un. Mat. It.} {\em 11-B} (1997) Suppl. Fasc., 257–288.

%\bibitem{GH} A. Gray, L.M. Hervella, The sixteen classes of almost Hermitian manifolds and thei linear invariants, {\em Ann. Mat. Pura Appl. } {\bf 123} (1980), 35-58.

%\bibitem{GHJ01} M. Gross, D. Huybrechts, D. Joyce, Calabi-Yau Manifolds and Related Geometries: Lectures at a Summer School in Nordfjordeid, Norway, Springer (2001).

\bibitem{HL} R. Harvey, J.B. Lawson, An intrinsic charactherization of K\"ahler manifolds, \emph{Inv. Math.} {\bf 74} (1983), 169–198.

\bibitem{Huy04} D. Huybrechts, {\em  Complex Geometry: An Introduction}, Springer-Verlag, Berlin, 2005.

\bibitem{Kas} H. Kasuya, Techniques of computations of Dolbeault cohomology of solvmanifolds, {\em Math. Z.} {\bf 273}, (2013), n. (1–2), 437–447.

\bibitem{KM} K. Kodaira, J. Morrow, {\em Complex Manifolds}, AMS Chelsea Publishing, 2006.

\bibitem{KS60} K. Kodaira, D.C. Spencer, On deformations of complex analytic structures, III. Stability
theorems for complex structures, {\em Ann. of Math.} {\bf 71} (1960), 43–76.

\bibitem{Kur65} K. Kuranishi, New Proof for the Existence of Locally Complete Families of Complex Structures, \emph{Proceedings of the Conference of Complex Analysis in Minneapolis}, Berlin: Springer-Verlag, (1965), 142-154.

%\bibitem{M62} A.I. Malcev, On a class of homogeneous spaces, {\em Amer. Soc. Transl. Ser.} {\bf 9} (1962), n. 1, 276-307.

%\bibitem{MPPS06} C. McLaughlin, H. Pedersen, Y.S. Poon, S. Salamon, Deformation of 2-step nilmanifolds with abelian complex structures, \emph{J.  Lond. Math. Soc.}, {\bf 2} (2006).

\bibitem{Mic} M.L. Michelson, On the existence of special metrics on complex geometry, {\em Acta. Math. } {\bf 149} (1982), 261-295.



%\bibitem{N} I. Nakamura, Complex parallelisable manifolds and their small deformations, {\em J. Differential Geometry}, {\bf 10} (1975), 85–112.



%\bibitem{DE} J.P. Demailly, {\em Complex Analytic and Differential Geometry}, Université de Grenoble, Saint-Martin d’Hères, 2012. 
%\bibitem{H} D. Huybrechts, {\em Complex Geometry. An Introduction}, Springer, Berlin 2005. 
%\bibitem{K} S. Kobayashi, {\em Differential Geometry of complex vector bundles}, Princeton University Press, Princeton, NJ, 2014.
%\bibitem{KS} K. Kodaira, D.C. Spencer, On deformations of complex analytic structures, III. Stability theorems for complex structures, {\em Ann. of Math.} {\bf 71} (1960), 43--76.

\bibitem{Nak} I. Nakamura, Complex parallelisable manifolds and their small deformations, {\em J. Differential Geom.} {\bf 10} (1975), 85-112.

%\bibitem{S} M. Schweitzer, Autour de la cohomologie de Bott-Chern, {\tt preprint} arXiv:0709.3528v1[math. AG]

\bibitem{OUV} A. Otal, L. Ugarte, R. Villacampa, Hermitian metrics on compact complex manifolds and their deformation limits,  in: Chiossi S., Fino A., Musso E., Podestà F., Vezzoni L. (eds) Special Metrics and Group Actions in Geometry, Springer INdAM Series, vol 23, Springer, Cham., (2017) 269-290.


\bibitem{PS21} R. Piovani, T. Sferruzza, Deformations of strong K\"ahler with torsion metrics, $\mathit{arXiv:2008.11983}$ [math.DG], (2021).

\bibitem{RWZ} S. Rao, X. Wan, Q. Zhao, On local stabilities of $p$-K\"ahler structures, {\em Compos. Math.} {\bf 155} (2019), n. 3, 455-483.

\bibitem{RWZ2} S. Rao, X. Wan, Q. Zhao, Power series proof for local stabilities of K\"ahler and balanced structures with mild $\del\delbar$-lemma, to be published in {\em Nagoya J. Math.} (2021), 1-50.

\bibitem{RZ} S. Rao, Q. Zao, Several special complex structures and their deformation properties, {\em J. Geom. Anal.} {\bf 28} (2018), 2984–3047.


%\bibitem{RT12} F.A. Rossi, A. Tomassini, On strong K\"ahler and astheno-K\"ahler metrics on nilmanifolds, \emph{Adv. Geom.} {\bf 12} (2012), 431-446.


\bibitem{ST} A. Saracco, A. Tomassini, On deformations of compact balanced manifolds, {\em Proc. Amer. Math. Soc.} {\bf 139} (2011), n. 2, 641-653. 

%\bibitem{S} M. Schweitzer, Autour de la cohomologie de Bott-Chern, $\mathit{arXiv:0709.3528v1}$ [math.AG], (2007).


%\bibitem{ST09} J. Streets, G. Tian, A Parabolic Flow of Pluriclosed Metrics Jeffrey Streets, {\em Int. Math. Res. Not. IMRN}, {\bf 2010}, 2010, n. 16, 3101–3133.

%\bibitem{Stre} J. Streets, Pluriclosed Flow and the Geometrization of Complex surfaces, in J. chen, P. Lu, Z. Lu, Z. Zhang (eds), Geometric Analysis, {\em Progr. Math.} {\bf 333}, 471-510. 

%\bibitem{ST12} J. Streets, G. Tian, Generalized K\"ahler geometry an pluriclosed flow, \emph{Nuclear Phys. B} {\bf 858} (2012), n. 2, 366-376.

%\bibitem{SU} J. Streets, Y. Ustinovskiy, Classification of Generalized K\"ahler-Ricci Solitons on Complex Surfaces, to appear in \emph{Comm. Pure Appl. Math.} {\bf 74} (2021).

%\bibitem{Stro} A. Strominger, Superstrings with torsion, \emph{Nuclear Phys. B} {\bf 274} (1986), 253–284.
 
\bibitem{Tu} J. Tu, The correspondance formula of Dolbeault complex on pair deformation, \emph{Geom. Dedicata} $\mathbf{212}$ (2021), 365-378. 
 
%\bibitem{U} L. Ugarte, Hermitian structures on six-dimensional nilmanifolds, {\em Transform. Groups}, {\bf 12}, (2007), 175–202.


%\bibitem{Wan} H.C. Wang, Complex parallelisable manifolds, {\em Proc. Amer. Math.} {\bf 5} (1954), 771-776.
\end{thebibliography}
\end{document}